\documentclass[10pt,a4paper]{amsart}
\calclayout
\usepackage{amsfonts}
\usepackage{enumerate}
\usepackage[english]{babel}
\usepackage{amsmath, amssymb, amsthm}
\usepackage{dsfont}
\usepackage{mathrsfs}
\usepackage{mathtools}
\usepackage[framemethod=TikZ]{mdframed}
\usepackage[justification=centering]{caption}
\usepackage{diagbox}
\usepackage[all]{xy}
\usepackage{multirow}
\usepackage{tikz}
\usetikzlibrary{
  matrix,
  calc,
  arrows
}

\newtheorem{thm}{Theorem}[section]
\newtheorem{lem}[thm]{Lemma}
\newtheorem{prop}[thm]{Proposition}
\newtheorem{cor}[thm]{Corollary}
\newtheorem*{thma}{Theorem A}
\newtheorem*{thmb}{Theorem B}
\newtheorem*{thmc}{Theorem C}

\theoremstyle{definition}
\newtheorem{defn}[thm]{Definition}

\newtheorem{exs}[thm]{Examples}
\newtheorem{rmk}[thm]{Remark}
\newtheorem{rmks}[thm]{Remarks}

\numberwithin{equation}{section}

\def\D{\mathbb D}

\def\N{\mathbb N}

\def\R{\mathbb R}
\def\Sp{\mathbb S} 

\def\Z{\mathbb Z}

\def\calH{\mathcal{H}}
\def\calI{\mathcal{I}}

\def\calL{\mathcal{L}}
\def\calM{\mathcal{M}}

\def\calO{\mathcal{O}}

\def\calU{\mathcal{U}}

\def\asr{\mathrm{asr}}
\def\Diff{\mathrm{Diff}}

\def\fr{\mathrm{fr}}
\def\GL{\mathrm{GL}}
\def\id{\mathrm{id}}
\def\Ker{\mathrm{Ker}}
\def\Link{\mathrm{Link}}
\def\Mod{\mathrm{Mod}}
\def\Quad{\mathrm{Quad}}
\def\rk{\mathrm{rk}}
\def\sr{\mathrm{sr}}
\def\st{\mathrm{st}}
\def\tn{\mathrm{tn}}
\def\usr{\mathrm{usr}}

\newcommand{\gbar}{\overline{g}}
\newcommand{\Kbar}{\overline{K}}
\newcommand{\rkbar}{\overline{\rk}}
\newcommand{\eps}{\varepsilon}

\begin{document}

\title{Homological Stability of automorphism groups of quadratic modules and manifolds}
\author{Nina Friedrich}
\email{N.Friedrich@maths.cam.ac.uk}
\address{Centre for Mathematical Sciences\\
Wilberforce Road\\
Cambridge CB3 0WB\\
UK}

\begin{abstract}
    We prove homological stability for both general linear groups of modules over a~ring with finite stable rank and unitary groups of quadratic modules over a~ring with finite unitary stable rank. In particular, we do not assume the modules and quadratic modules to be well-behaved in any sense: for example, the quadratic form may be singular. This extends results by van der Kallen and Mirzaii--van der Kallen respectively. Combining these results with the machinery introduced by Galatius--Randal-Williams to prove homological stability for moduli spaces of simply-connected manifolds of dimension~$2n \geq 6$, we get an~extension of their result to the case of virtually polycyclic fundamental groups.
\end{abstract}

\maketitle

\section{Introduction and Statement of Results}

We say that the sequence $X_1 \overset{f_1}{\longrightarrow} X_2 \overset{f_2}{\longrightarrow} X_3 \overset{f_3}{\longrightarrow} \cdots$ of topological spaces satisfies homological stability if the induced maps $(f_k)_* \colon H_k(X_n) \longrightarrow H_k(X_{n+1})$ are isomorphisms for $k < An + B$ for some constants~$A$ and $B$. In most cases where homological stability is known it is extremely hard to compute any particular $H_k(X_n)$. However, there are several techniques to compute the stable homology groups $H_k(X_{\infty})$ and homological stability can therefore be used to give many potentially new homology groups.

\subsection{General Linear Groups}

In~\cite{vdK}, van der Kallen proves homological stability for the group $\GL_n(R)$ of $R$-module automorphisms of $R^n$. For the special case where $R$ is a~PID, Charney~\cite{CharneyDedekindDomain} had earlier shown homological stability. In the first part of this paper we consider the analogous homological stability problem for groups of automorphisms of general $R$-modules~$M$; we write $\GL(M)$ for these groups. In order to phrase our stability range we define the rank of an~$R$-module~$M$, $\rk(M)$, to be the biggest number~$n$ so that $R^n$ is a~direct summand of~$M$. The stability range then says that the rank of $M$ has to be big compared to the so-called stable rank of~$R$, $\sr(R)$. In particular, the stable rank of~$R$ needs to be finite which holds for example for Dedekind domains and more generally algebras that are finite as a~module over a~commutative Noetherian ring of finite Krull dimension.

\begin{thma}
    The map
        \[H_k(\GL(M); \Z) \rightarrow H_k(\GL(M \oplus R); \Z),\]
    induced by the inclusion $\GL(M) \hookrightarrow \GL(M \oplus R)$, is an~epimorphism for $k \leq \frac{\rk(M) - \sr(R)}{2}$ and an~isomorphism for $k \leq \frac{\rk(M) - \sr(R) - 1}{2}$.

    For the commutator subgroup $\GL(M)^{\prime}$ the map
        \[H_k(\GL(M)^{\prime}; \Z) \rightarrow H_k(\GL(M \oplus R)^{\prime}; \Z)\]
    is an~epimorphism for $k \leq \frac{\rk(M) - \sr(R) - 1}{3}$ and an~isomorphism for $k \leq \frac{\rk(M) - \sr(R) - 3}{3}$.
\end{thma}

We emphasise that $M$ is allowed to be any module over~$R$. For example over the integers, $M$ could be $\Z/100\Z \oplus \Z^{100}$. We also get statements for polynomial and abelian coefficients. The full statement of our theorem is given in Theorem~\ref{ThmMainGL}.

This part of the paper can be seen as a~warm up for the heart of the algebraic part of this paper, which is homological stability for the automorphism groups of quadratic modules.

\subsection{Unitary Groups}

A~quadratic module is a~tuple $(M, \lambda, \mu)$ consisting of an~$R$-module~$M$, a~sesquilinear form~$\lambda \colon M \times M \rightarrow R$, and a~function~$\mu$ on~$M$ into a~quotient of~$R$, where $\lambda$ measures how far $\mu$ is from being linear. The precise definition is given in Section~\ref{SecCpxU}.
The basic example of a~quadratic module is the \textit{hyperbolic module}~$H$, which is given by
    \[\left( R^2 \text{ with basis } e, f; \begin{pmatrix} 0 & 1 \\ \eps & 0 \end{pmatrix}; \mu \text{ determined by } \mu(e) = \mu(f) = 0 \right).\]
For a~quadratic module~$M$ we write $U(M)$ for its unitary group, i.e.\ the group of all automorphisms that fix the quadratic structure on $M$. Mirzaii--van der Kallen~\cite{MvdK} have shown homological stability for the unitary groups~$U(H^n)$ and our Theorem~B below extends this to general quadratic modules.

We write $g(M)$ for the Witt index of~$M$ as a~quadratic module, which is defined to be the maximal number~$n$ so that $H^n$ is a~direct summand of $M$. In our stability range we use the notion of unitary stable rank of~$R$, $\usr(R)$, which is at least as big as the stable rank and also requires a~certain transitivity condition on unimodular vectors of fixed length. Analogously to Theorem~A the Witt index of~$M$ has to be big in relation to the unitary stable rank of~$R$. In particular, $\usr(R)$ needs to be finite which is the case for both examples given above of rings with finite stable rank.

\begin{thmb}
    The map
        \[H_k(U(M); \Z) \rightarrow H_k(U(M \oplus H); \Z)\]
    is an~epimorphism for $k \leq \frac{g(M) - \usr(R) - 1}{2}$ and an~isomorphism for $k < \frac{g(M) - \usr(R) - 2}{2}$.

    For the commutator subgroup $U(M)^{\prime}$ the map
        \[H_k(U(M)^{\prime}; \Z) \rightarrow H_k(U(M \oplus H)^{\prime}; \Z)\]
    is an~epimorphism for $k \leq \frac{g(M) - \usr(R) - 1}{2}$ and an~isomorphism for $k < \frac{g(M) - \usr(R) - 3}{2}$.
\end{thmb}

We again emphasise that $M$ can be an~arbitrary quadratic module -- in particular, it can be singular. As in the case for general linear groups, we get an~analogous statement for abelian and polynomial coefficients. The full statement is given in Theorem~\ref{ThmMainU}.

To show homological stability for both the automorphism groups of modules and quadratic modules we use the machinery developed in Randal-Williams--Wahl~\cite{R-WWahl}. The actual homological stability results are straightforward applications of that paper assuming that a~certain semisimplicial set is highly connected. Showing that this assumption is indeed satisfied is the main goal in Chapters~\ref{ChHomStabGL} and \ref{ChHomStabU}.

\subsection{Moduli Spaces of Manifolds}

Our theorem in the unitary case can also be used to extend the homological stability result for moduli spaces of simply-connected manifolds of dimension~$2n \geq 6$ by Galatius--Randal-Williams~\cite{GR-WStabNew} to certain non-simply-connected manifolds.

For a~compact connected smooth manifold~$W$ of dimension~$2n$ we write $\Diff_{\partial}(W)$ for the topological group of all diffeomorphisms of~$W$ that restrict to the identity near the boundary, and call its classifying space $B\Diff_{\partial}(W)$ the \textit{moduli space of manifolds of type~$W$}. As in the algebraic settings described previously there is a~notion of rank: Define the \textit{genus} of~$W$ as
    \[g(W) := \sup\{g \in \N \ | \ \text{there are $g$ disjoint embedding of~$\Sp^n \times \Sp^n \setminus \mathrm{int}(\D^{2n})$ into } W\}.\]
Let $S$ denote the manifold $\left([0, 1] \times \partial W \right) \# \left(\Sp^n \times \Sp^n\right)$. We get an~inclusion
    \[\Diff_{\partial}(W) \hookrightarrow \Diff_{\partial}(W \cup_{\partial W} S)\]
by extending diffeomorphisms by the identity on~$S$. This gluing map then has an~induced map on classifying spaces which we denote by~$s$. Galatius--Randal-Williams have shown that for simply-connected manifolds of dimension $2n \geq 6$ the induced map
    \[s_* \colon H_k(B\Diff_{\partial}(W)) \longrightarrow H_k(B\Diff_{\partial}(W \cup_{\partial W} S))\]
is an~epimorphism for $k \leq \frac{g(W) - 1}{2}$ and an~isomorphism for $k \leq \frac{g(W) - 3}{2}$. The following extends this result to certain non-simply-connected manifolds.

\begin{thmc}
    Let $W$ be a~compact connected manifold of dimension $2n \geq 6$. Then the map
        \[s_* \colon H_k(B\Diff_{\partial}(W)) \longrightarrow H_k(B\Diff_{\partial}(W \cup_{\partial W} S))\]
    is an~epimorphism for $k \leq \frac{g(W) - \usr(\Z[\pi_1(W)])}{2}$ and an~isomorphism for $k \leq \frac{g(W) - \usr(\Z[\pi_1(W)]) - 2}{2}$.
\end{thmc}

For a~virtually polycyclic fundamental group, e.g.\ a~finitely generated abelian group, the unitary stable rank of its group ring is known to be finite by Crowley-Sixt~\cite{CrowleySixt}. Combining Theorem~C with \cite[Cor.~1.9]{GR-WStab2} yields a~computation of~$H_k(B\Diff_{\partial}(W))$ in the stable range.

\subsection*{Acknowledgements}

These results will form part of my Cambridge PhD thesis. I am grateful to my supervisor Oscar Randal-Williams for many interesting and inspiring conversations and much helpful advice. I would like to thank the anonymous referee for pointing out a~gap in an~earlier version of this paper and their valuable input towards solving this. I was partially supported by the ``Studienstiftung des deutschen Volkes'' and by the EPSRC.

\section{Homological Stability for General Linear Groups} \label{ChHomStabGL}

This chapter treats the case of automorphism groups of modules. For the case of modules of the form $R^n$ for some ring~$R$ there are several results available already, e.g.\ results by Charney~\cite{CharneyDedekindDomain} for $R$ a~Dedekind domain and by van der Kallen~\cite{vdK} for $R$ with finite stable rank.

We consider the case of general modules over a~ring with finite stable rank. The approach we use to show homological stability is what has become the standard strategy of proving results in this area. It has been introduced by Maazen~\cite{Maazen} and shortly afterwards used in various contexts by Charney~\cite{CharneyDedekindDomain}, Dwyer~\cite{DwyerTwistedHomStab}, van der Kallen~\cite{vdK}, and Vogtmann~\cite{Vogtmann}. For us it is convenient to use the formulation in Randal-Williams--Wahl~\cite{R-WWahl}. This mainly involves showing the high connectivity of a~certain semisimplicial set. We start by generalising a~complex introduced by van der Kallen and show its high connectivity. Even though this complex is not exactly the one needed for the machinery of Randal-Williams--Wahl, it is good enough to deduce the high connectivity of that semisimplicial set. We can then immediately extract a~homological stability result for various coefficients systems.

\subsection{The Complex and its Connectivity} \label{SecCpxGL}

Following~\cite{vdK}, for a~set~$V$ we define $\calO(V)$ to be the poset of ordered sequences of distinct elements in~$V$ of length at least one. The partial ordering on $\calO(V)$ is given by refinement, i.e.\ we write $(w_1, \ldots, w_m) \leq (v_1, \ldots, v_n)$ if there is a~strictly increasing map $\phi \colon \{1, \ldots, m\} \rightarrow \{1, \ldots, n\}$ such that $w_i = v_{\phi(i)}$. We say that $F \subseteq \calO(V)$ satisfies the \textit{chain condition} if for every element $(v_1, \ldots, v_n) \in F$ and every $(w_1, \ldots, w_m) \leq (v_1, \ldots, v_n)$ we also have $(w_1, \ldots, w_m) \in F$. For $v = (v_1, \ldots, v_n) \in F$, we write $F_v$ for the set of all sequences $(w_1, \ldots, w_m) \in F$ such that $(w_1, \ldots, w_m, v_1, \ldots, v_n) \in F$. Note that if $F$ satisfies the chain condition and $v, w \in F$ then $(F_v)_w = F_{vw}$. We write $F_{\leq k}$ for the subset of~$F$ containing all sequences of length $\leq k$.

We write $\GL(M)$ for the group of automorphisms of general $R$-modules~$M$. A~sequence $(v_1, \ldots, v_n)$ of elements in~$M$ is called \textit{unimodular} if there are $R$-module homomorphisms
    \[f_1, \ldots, f_n \colon R \rightarrow M \text{ and } \phi_1, \ldots, \phi_n \colon M \rightarrow R\]
such that $f_i(1) = v_i$ and $\phi_j \circ f_i = \delta_{i, j} \cdot \mathds{1}_R$. An~element $v \in M$ is called \textit{unimodular} if it is unimodular as a~sequence in~$M$ of length~$1$. The condition $\phi_j \circ f_i = \delta_{i, j} \cdot \mathds{1}_R$ holds if and only if the matrix $\left(\phi_j \circ f_i (1) \right)_{i, j}$ is the identity matrix. In fact, for a~sequence to be unimodular it is enough to find $\tilde \phi_1, \ldots, \tilde \phi_n$ so that the matrix $\left(\tilde \phi_j \circ f_i (1) \right)_{i, j}$ is invertible.

\begin{lem} \label{LemUnimodularInvertible}
    Given a~sequence $(v_1, \ldots, v_n)$ in~$M$ and $R$-module homomorphisms
        \[f_1, \ldots, f_n \colon R \rightarrow M \text{ and } \tilde \phi_1, \ldots, \tilde \phi_n \colon M \rightarrow R\]
    so that $f_i(1) = v_i$  and the matrix $\left(\tilde \phi_j \circ f_i (1) \right)_{i, j}$ is invertible. Then $(v_1, \ldots, v_n)$ is already unimodular.
\end{lem}

\begin{proof}
    Let $A^{-1}$ denote the inverse of the matrix $\left(\tilde \phi_j \circ f_i (1) \right)_{i, j}$. We define $R$-module homomorphisms $\phi_j \colon M \rightarrow R$ as follows:
        \[\phi_1 \oplus \cdots \oplus \phi_n \colon M \xrightarrow{\tilde \phi_1 \oplus \cdots \oplus \tilde \phi_n} R^n \xrightarrow{\cdot A^{-1}} R^n,\]
    where $\phi_j(m)$ is the $j$-th entry of the vector $\phi_1 \oplus \cdots \oplus \phi_n (m)$. By construction we have $\phi_j (v_i) = \delta_{i, j}$ and therefore the sequence $(v_1, \ldots, v_n)$ is unimodular.
\end{proof}

Let $R^{\infty}$ denote the free $R$-module with basis $e_1, e_2, \ldots$ and let $M^{\infty}$ denote the $R$-module $M \oplus R^{\infty}$. Then we write $\calU(M)$ for the subposet of~$\calO(M)$ consisting of unimodular sequences in~$M$. Note that for $(v_1, \ldots, v_n) \in M$ it is the same to say the sequence is unimodular in M or it is unimodular in $M \oplus R^{\infty}$.

\begin{defn} \label{DefStableRank}
A~ring~$R$ satisfies the \textit{stable range condition}~$(\mathrm{S}_n)$ if for every unimodular vector $(r_1, \ldots, r_{n+1}) \in R^{n+1}$ there are $t_1, \ldots, t_n \in R$ such that the vector $(r_1 + t_1 r_{n+1}, \ldots, r_n + t_n r_{n+1}) \in R^n$ is unimodular. If~$n$ is the smallest such number we say~$R$ has \textit{stable rank}~$n$, $\sr(R) = n$ and it has $\sr(R)= \infty$ if such an~$n$ does not exist.
\end{defn}

Note that the stable range in the sense of Bass~\cite{Bass}, $(\mathrm{SR}_n)$, is the same as our stable range condition~$(\mathrm{S}_{n-1})$. The absolute stable rank of a~ring~$R$, $\asr(R)$, as defined by Magurn--van der Kallen--Vaserstein in~\cite{MvdKV} is an~upper bound for the stable rank, i.e.\ $\sr(R) \leq \asr(R)$ (\cite[Lemma~1.2]{MvdKV}). In the following we give some of the well-known examples of rings and their stable ranks.

\begin{exs} \label{ExsSr}
    \leavevmode
    \begin{enumerate}
        \item A~commutative Noetherian ring~$R$ of finite Krull dimension $d$ satisfies $\sr(R) \leq d + 1$. In particular, if $R$ is a~Dedekind domain then $\sr(R) \leq 2$ (\cite[4.1.11]{HahnOMeara}) and for a~field~$k$, the polynomial ring $K = k[t_1, \ldots, t_n]$ satisfies $\sr(K) \leq n + 1$ (\cite[Thm.~8]{VasersteinDim}).
        \item More generally, any $R$-algebra~$A$ that is finitely generated as an~$R$-module satisfies $\sr(A) \leq d + 1$, for $R$ again a~commutative Noetherian ring of finite Krull dimension~$d$. \cite[Thm.~3.1]{MvdKV} or \cite[4.1.15]{HahnOMeara}
        \item Recall that a~ring~$R$ is called~\textit{semi-local} if $R/J(R)$ is a~left Artinian ring, for $J(R)$ the Jacobson radical of~$R$. A~semi-local ring satisfies $\sr(R) =  1$. \cite[4.1.17]{HahnOMeara}
        \item Recall that a~group~$G$ is called \textit{virtually polycyclic} if there is a~sequence of normal subgroups
                \[G = G_0 \triangleright G_1 \triangleright \ldots \triangleright G_{n-1} \triangleright G_n = 0\]
            such that each quotient $G_i/G_{i+1}$ is cyclic or finite. Its \textit{Hirsch number}~$h(G)$ is the number of infinite cyclic factors. For a~virtually polycyclic group~$G$ we have $\sr(\Z[G]) \leq h(G) + 2$. \cite[Thm.~7.3]{CrowleySixt}
    \end{enumerate}
\end{exs}

For an~$R$-module~$M$ we define the \textit{rank} of~$M$ as
    \[\rk(M) := \sup \{n \in \N \ | \ \text{there is an~$R$-module $M^{\prime}$ such that } M \cong R^n \oplus M^{\prime} \}.\]
Using this notion we can phrase the following theorem. Here and in the following, we use the convention that the condition of a~space to be $n$-connected for $n \leq -2$ (and so in particular for $n = -\infty$) is vacuous.

\begin{thm} \label{ThmGLConnectivity}
    \leavevmode
    \begin{enumerate}
        \item $\calO(M) \cap \calU(M^{\infty})$ is $(\rk(M) - \sr(R) - 1)$-connected,
        \item $\calO(M) \cap \calU(M^{\infty})_{(v_1, \ldots, v_k)}$ is $(\rk(M) - \sr(R) - k - 1)$-connected for $(v_1, \ldots, v_k) \in \calU(M^{\infty})$.
    \end{enumerate}
\end{thm}

In~\cite[Thm.~2.6~(i), (ii)]{vdK} van der Kallen has proven this theorem for the special case of modules of the form $R^n$. Our proof of Theorem~\ref{ThmGLConnectivity} adapts the techniques and ideas that he has used. Just as in van der Kallen's proof, we use the following technical lemma several times in the proof of Theorem~\ref{ThmGLConnectivity}.

\begin{lem} \label{Lem2.13vdK}
    Let $F \subseteq \calU(M^{\infty})$ satisfy the chain condition. Let $X \subseteq M^{\infty}$ be a~subset.
    \begin{enumerate}
        \item Assume that the poset $\calO(X) \cap F$ is $d$-connected and that, for all sequences $(v_1, \ldots, v_m)$ in $F \setminus \calO(X)$, the poset $\calO(X) \cap F_{(v_1, \ldots, v_m)}$ is $(d-m)$-connected. Then $F$ is $d$-connected.
        \item Assume that for all sequences $(v_1, \ldots, v_m)$ in $F \setminus \calO(X)$, the poset $\calO(X) \cap F_{(v_1, \ldots, v_m)}$ is $(d-m+1)$-connected. Assume further that there is a~sequence~$(y_0)$ of length~$1$ in~$F$ with $\calO(X) \cap F \subseteq F_{(y_0)}$. Then $F$ is $(d+1)$-connected.
    \end{enumerate}
\end{lem}

\begin{proof}[Outline of the proof.]
    The proof of~\cite[Lemma~2.13]{vdK} also works in this setting, where we use the obvious modification of~\cite[Lemma~2.12]{vdK} to allow $F \subseteq \calU(M^{\infty})$ so that it fits into our framework.
\end{proof}

We are not the first ones that have the idea of showing homological stability for automorphism groups of modules more general than $R^n$: In~\cite[Rmk.~2.7~(2)]{vdK} van der Kallen has suggested a~possible generalisation of his results using the notion of ``big'' modules as defined in~\cite{VasersteinClassicalGroups}.

\begin{proof}[Proof of Theorem~\ref{ThmGLConnectivity}]
    Analogous to the proof of~\cite[Thm.~2.6]{vdK} we will also show the following statements.
    \begin{enumerate}[\hspace{0.65cm}(a)]
        \item $\calO(M \cup (M + e_1)) \cap \calU(M^{\infty})$ is $(\rk(M) - \sr(R))$-connected,
        \item $\calO\left(M \cup (M + e_1)\right) \cap \calU(M^{\infty})_{(v_1, \ldots, v_k)}$ is $(\rk(M) - \sr(R) - k)$-connected for $(v_1, \ldots, v_k) \in \calU(M^{\infty})$.
    \end{enumerate}
    Recall that $e_1$ denotes the first standard basis element of $R^{\infty}$ in $M^{\infty} = M \oplus R^{\infty}$.

    The proof is by induction on~$g = \rk(M)$. Note that statements~(1), (2), and (b) all hold for $g < \sr(R)$ so we can assume $g \geq \sr(R)$. Statement~(a) holds for $g < \sr(R) - 1$ so we can assume $g \geq \sr(R) - 1$ when proving this statement. The structure of the proof is as follows. We start by proving (b) which enables us to deduce (2). We will then prove statements~(1) and (a) simultaneously by applying statement~(2).

    We may suppose $M = R^g \oplus M^{\prime}$ for an~$R$-module~$M^{\prime}$, since the posets in statements~(1), (2), (a), and (b) only depend on the isomorphism class of~$M$. We write $x_1, \ldots, x_g$ for the standard basis of $R^g$.

    \underline{Proof of (b).} For $Y := M \cup (M + e_1)$ we write $F := \calO(Y) \cap \calU(M^{\infty})_{(v_1, \ldots, v_k)}$. Let $d := g - \sr(R) - k$, so we have to show that $F$ is $d$-connected.

    In the case $g = \sr(R)$ we only have to consider $k = 1$. Then we have to show that $F$ is non-empty. The strategy for this part is as follows: We define a~map $f \in \GL(M^{\infty})$ so that $Y$ is fixed under~$f$ as a~set and the projection of $f(v_1)$ onto $R^g$, $f(v_1)|_{R^g}$, is unimodular. Then the sequence $(f(v_1)|_{R^g}, e_1)$ is unimodular in~$M^{\infty}$. We will show that, therefore, the sequence $(f(v_1), e_1)$ is also unimodular in~$M^{\infty}$ and so is the sequence $(v_1, f^{-1}(e_1))$. Since $e_1 \in Y$ and the automorphism~$f$ fixes $Y$ setwise we get $f^{-1}(e_1) \in Y$ and thus $F$ is non-empty as it contains~$f^{-1}(e_1)$.

    We start by writing
        \[v_1 = \sum^g_{i = 1} x_i r_i + p + a,\]
    where $r_i \in R$, $p \in M^{\prime}$, and $a \in R^{\infty}$. Since $v_1$ is unimodular there is an~$R$-module homomorphism $\phi \colon M^{\infty} \rightarrow R$ satisfying $\phi(v_1) = 1$. In particular,
        \[1 = \phi(v_1) = \sum^g_{i=1} \phi(x_i) r_i + \phi(p + a),\]
    which shows that $(r_1, \ldots, r_g, \phi(p + a)) \in R^{g + 1}$ is unimodular. As $g = \sr(R)$ there are $t_1, \ldots, t_g \in R$ such that the sequence
        \[(r_1 + t_1 \phi(p + a), \ldots, r_g + t_g \phi(p + a))\]
    is unimodular. Now consider the map
    \begin{eqnarray*}
        M^{\infty} = R^g \oplus M^{\prime} \oplus R^{\infty} \!
        & \! \overset{f}{\longrightarrow}
        & M^{\infty} = R^g \oplus M^{\prime} \oplus R^{\infty} \\
        (a_1, \ldots, a_g, q, b)
        & \longmapsto \!
        & \! (a_1 + t_1 \phi(q + b), \ldots, a_g + t_g \phi(q + b), q, b),
    \end{eqnarray*}
    which is invertible.
    The map~$f$ satisfies $f(Y) = Y$ and the projection of $f(v_1)$ onto $R^g$ is unimodular. Thus, by definition there are homomorphisms $f_1 \colon R \rightarrow M^{\infty}$ and $\phi_1 \colon M^{\infty} \rightarrow R$ so that $f_1(1) = f(v_1)|_{R^g}$ and $\phi_1 \circ f_1 = \mathds{1}_R$. Note that we can assume that $\phi_1$ is zero away from $R^g$ as otherwise we can restrict to $R^g$ before we apply $\phi_1$. This shows that the sequence $(f(v_1)|_{R^g}, e_1)$ is unimodular by choosing $\phi_2 \colon M^{\infty} \rightarrow R$ to be the projection onto the coefficient of~$e_1$. For the sequence $(f(v_1), e_1)$ we change $f_1$ to map $1$ to $f(v_1)$ but keeping all other homomorphisms the same then the matrix $\left(\tilde \phi_j \circ f_i (1) \right)_{i, j}$ is an~upper triangular matrix with $1$'s on the diagonal. In particular, it is invertible, so the sequence $(f(v_1), e_1)$ is unimodular by Lemma~\ref{LemUnimodularInvertible}. Since $f$ is an~automorphism of~$M^{\infty}$ the sequence $(v_1, f^{-1}(e_1))$ is also unimodular. By construction we have $f(Y) = Y$ and so in particular $f^{-1}(e_1) \in Y$. Hence, $F$ is non-empty as it contains $f^{-1}(e_1)$.

    Now consider the case $g > \sr(R)$. As in the case above there is an~$f \in \GL(M^{\infty})$ such that $f(Y) = Y$ and $f(v_1)|_{R^g}$ is unimodular. The group~$\GL_g(R)$ acts transitively on the set of unimodular elements in~$R^g$ (by~\cite[Thm.~2.3 (c)]{VasersteinGL}). This only holds in the case $g > \sr(R)$ so the case $g = \sr(R)$ had to be proven separately. Hence, there exists a~map $\psi \in \GL_g(R) \leq \GL(M^{\infty})$ such that $\psi(f(v_1)|_{R^g}) = x_g$.
    By applying $\psi \circ f$, considered as an~automorphism of~$M^{\infty}$, to $M^{\infty}$, without loss of generality we can assume that the projection of~$v_1$ to~$R^g$ is $x_g$. We define
        \[X := \{v \in Y \ | \ \text{the $x_g$-coordinate of~$v$ vanishes}\} = (R^{g-1} \oplus M^{\prime}) \cup (R^{g-1} \oplus M^{\prime} + e_1).\]
    We now check that the assumptions of Lemma~\ref{Lem2.13vdK}~(1) are satisfied. Notice that
        \[\calU(M^{\infty})_{(v_1, \ldots, v_k)} = \calU(M^{\infty})_{(v_1, v^{\prime}_2, \ldots, v^{\prime}_k)},\]
    for $v^{\prime}_i = v_i + v_1 \cdot r_i$ for $r_i \in R$, as the span of $v_1, v_2^{\prime}, \ldots, v_k^{\prime}$ is the same as that of $v_1, v_2, \ldots, v_n$. As the projection of $v_1$ to~$R^g$ is $x_g$, we may choose the $r_i$ so that the $x_g$-coordinate of each $v_i^{\prime}$ vanishes.
    \begin{eqnarray*}
        \calO(X) \cap F
        & = & \calO(X) \cap \calO(Y) \cap \calU(M^{\infty})_{(v_1, \ldots, v_k)} \\
        & = & \calO((R^{g-1} \oplus M^{\prime}) \cup (R^{g-1} \oplus M^{\prime} + e_1)) \cap \calU(M^{\infty})_{(v^{\prime}_2, \ldots, v^{\prime}_k)}.
    \end{eqnarray*}
    Therefore, by the induction hypothesis, $\calO(X) \cap F$ is $d$-connected. Analogously, for $(w_1, \ldots, w_l) \in F \setminus \calO(X)$ we get
    \begin{eqnarray*}
        \calO(X) \cap F_{(w_1, \ldots, w_l)}
        & = & \calO(X) \cap \calO(Y) \cap \calU(M^{\infty})_{(v_1, \ldots, v_k, w_1, \ldots, w_l)} \\
            & = & \calO((R^{g-1} \oplus M^{\prime} \cup (R^{g-1} \oplus M^{\prime} + e_1))) \cap \calU(M^{\infty})_{(v^{\prime}_2, \ldots, v^{\prime}_k, w^{\prime}_1, \ldots, w^{\prime}_l)},
    \end{eqnarray*}
    which is $(d-l)$-connected by the induction hypothesis. Therefore, Lemma~\ref{Lem2.13vdK}~(1) shows that $F$ is $d$-connected.

    \underline{Proof of (2).} Let us write
        \[X := \left(R^{g - 1} \oplus M^{\prime}\right) \cup \left((R^{g - 1} + x_g) \oplus M^{\prime}\right).\]
    Then we have
    \begin{eqnarray*}
        & & \calO(X) \cap \left(\calO(M) \cap \calU(M^{\infty})_{(v_1, \ldots, v_k)}\right) \\
        & = & \calO\Big((R^{g - 1} \oplus M^{\prime}) \cup \big((R^{g - 1} + x_g) \oplus M^{\prime}\big)\Big) \cap \calU(M^{\infty})_{(v_1, \ldots, v_k)},
    \end{eqnarray*}
    which is $(d - k -1)$-connected by (b) after a~change of coordinates.

    Similarly, for $(w_1, \ldots, w_l) \in \calO(M) \cap \calU(M^{\infty})_{(v_1, \ldots, v_k)} \setminus \calO(X)$ we have
    \begin{eqnarray*}
        & & \calO(X) \cap \left(\calO(M) \cap \calU(M^{\infty})_{(v_1, \ldots, v_k)}\right)_{(w_1, \ldots, w_l)} \\
        & = & \calO(X) \cap \left(\calO(M) \cap \calU(M^{\infty})_{(v_1, \ldots, v_k, w_1, \ldots, w_l)}\right),
    \end{eqnarray*}
    which is $(d - k - l - 1)$-connected by the above. Hence, by Lemma~\ref{Lem2.13vdK}~(1) the claim follows.

    \underline{Proof of (1) and (a).} Recall that we now only assume $g \geq \sr(R) - 1$. By induction let us assume that statement~(a) holds for $R^{g - 1} \oplus M^{\prime}$ and we want to deduce it for $M = R^g \oplus M^{\prime}$. Before we finish the induction for~(a) we will show that this already implies statement~(1) for $M = R^g \oplus M^{\prime}$. For this consider $X$ to be as in the proof of~(2) and $d := g - \sr(R)$. Then
    \begin{eqnarray*}
        & & \calO(X) \cap \left(\calO(M) \cap \calU(M^{\infty})\right) \\
        & = & \calO\Big((R^{g - 1} \oplus M^{\prime}) \cup \big((R^{g - 1} + x_g) \oplus M^{\prime}\big)\Big) \cap \calU(M^{\infty})
    \end{eqnarray*}
    is $(d - 1)$-connected by (a) after a~change of coordinates. The remaining assumption of Lemma~\ref{Lem2.13vdK}~(1), i.e.\ that $\calO(X) \cap \left(\calO(M) \cap \calU(M^{\infty})\right)_{(v_1, \ldots, v_m)}$ is $(d - m - 1)$-connected, we have already shown in the proof of~(2). Thus, $\calO(M) \cap \calU(M^{\infty})$ is $(g - \sr(R) - 1)$-connected which proves statement~(1).

    To prove (a) we will apply Lemma~\ref{Lem2.13vdK}~(2) for $X = M$ and $y_0 = e_1$. Consider
        \[(v_1, \ldots, v_k) \in \calO(M \cup (M + e_1)) \cap \calU(M^{\infty}) \setminus \calO(X).\]
    Without loss of generality we may suppose that $v_1 \notin X$ as otherwise we can permute the $v_i$. By definition of~$X$ the coefficient of the $e_1$-coordinate of~$v_1$ is therefore $1$. Analogous to the proof of~(b) we have
        \[\calO(X) \cap \calO(M \cup (M + e_1)) \cap \calU(M^{\infty})_{(v_1, \ldots, v_k)} \cong \calO(M) \cap \calU(M^\infty)_{(v^{\prime}_2, \ldots, v^{\prime}_k)},\]
    where $v^{\prime}_i := v_i + v_1 r_i$ is chosen so that the $e_1$-coordinate of~$v^{\prime}_i$ is~$0$ for all~$i$. This is $(d - k)$-connected by~(1) for $k = 1$ and by~(2) for $k \geq 2$. By construction we have
        \[\calO(X) \cap \calO(M \cup (M + e_1)) \cap \calU(M^{\infty}) \subseteq (\calO(M \cup (M + e_1)) \cap \calU(M^{\infty}))_{(e_1)}\]
    and thus we can apply Lemma~\ref{Lem2.13vdK}~(2) to show that $\calO(M \cup (M + e_1)) \cap \calU(M^{\infty})$ is $(g - \sr(R))$-connected which proves~(a).

    When showing statement~(a) for $M = R^g \oplus M^{\prime}$ we only used statement~(1) for $M = R^g \oplus M^{\prime}$ which follows from (a) for $R^{g-1} \oplus M^{\prime}$ so this is indeed a~valid induction to show both statements~(1) and (a).
\end{proof}

The following propositions are consequences of the path-connectedness of $\calO(M) \cap \calU(M^{\infty})$ and therefore, by Theorem~\ref{ThmGLConnectivity}, hold in particular for $R$-modules~$M$ such that $\rk(M) \geq \sr(R) + 1$. The statements and proofs are \cite[Prop.~3.3]{GR-WStabNew} and \cite[Prop.~3.4]{GR-WStabNew} respectively for the case of general $R$-modules.

\begin{prop}[Transitivity] \label{PropGLTransitivity}
    If $\phi_0, \phi_1 \colon R \rightarrow M$ are split injective morphisms of $R$-modules and the poset $\calO(M) \cap \calU(M^{\infty})$ is path-connected, then there is an~automorphism~$f$ of~$M$ such that $\phi_1 = f \circ \phi_0$.
\end{prop}

\begin{proof}
    Note that an~$R$-module map $R \rightarrow M$ is defined by where it sends the unit~$1$ of the ring~$R$. Suppose first that $(\phi_1(1), \phi_2(1))$ is in $\calO(M) \cap \calU(M^{\infty})$. This implies
        \[M \cong \phi_1(R) \oplus \phi_2(R) \oplus M^{\prime}\]
    for some $R$-module~$M^{\prime}$ and that there is an~automorphism of~$M$ which interchanges the $\phi_i(R)$ and fixes~$M^{\prime}$. Consider the equivalence relation between morphisms $f \colon R \rightarrow M$ of differing by an~automorphism of~$M$. We have just shown that two morphisms corresponding to two adjacent vertices in $\calO(M) \cap \calU(M^{\infty})$ are equivalent. But the poset is path connected by assumption and hence all vertices are equivalent.
\end{proof}

\begin{prop}[Cancellation] \label{PropGLCancellation}
    Let $M$ and $N$ be $R$-modules with $M \oplus R \cong N \oplus R$. If the poset $\calO(M \oplus R) \cap \calU(M^{\infty})$ is path-connected, then there is also an~isomorphism $M \cong N$.
\end{prop}

\begin{proof}
    As in the proof of Proposition~\ref{PropGLTransitivity} we can assume that the isomorphism $\phi \colon M \oplus R \rightarrow N \oplus R$ satisfies $\phi|_R = \id_R$. Thus, by considering quotient modules we get
        \[ M \cong \frac{M \oplus R}{R} \cong \frac{\phi(M \oplus R)}{\phi(R)} = \frac{N \oplus R}{R} \cong N. \qedhere \]
\end{proof}

\subsection{Homological Stability} \label{SecHomStabGL}

We now prove homological stability of general linear groups over modules (Theorem~\ref{ThmMainGL}), which induces in particular Theorem~A, using the machinery of Randal-Williams--Wahl~\cite{R-WWahl}. We write $(fR \text{-} \Mod, \oplus, 0)$ for the groupoid of finitely-generated right $R$-modules and their isomorphisms. In order to apply the main homological stability theorems in~\cite{R-WWahl} we need to show that the corresponding category $UfR\text{-}\Mod := \langle fR\text{-}\Mod, fR\text{-}\Mod \rangle$ defined in~\cite[Sec.~1.1]{R-WWahl} satisfies the required axioms, i.e.\ it is locally homogeneous and satisfies the connectivity axiom LH3. Note that local homogeneity at $(M, R)$ for an~$R$-module~$M$ satisfying $\rk(M) \geq \sr(R)$ follows from \cite[Prop.~1.6]{R-WWahl} and \cite[Thm.~1.8~(a), (b)]{R-WWahl}. The following lemma verifies the axiom~LH3 from the connectivity of the complex considered in Theorem~\ref{ThmGLConnectivity}.

\begin{lem} \label{LemLH3GL}
    The semisimplicial set $W_n(M, R)_{\bullet}$ as defined in~\cite[Def.~2.1]{R-WWahl} is $\bigl \lfloor \frac{n + \rk(M) - \sr(R) - 2}{2}\bigr \rfloor$-connected.
\end{lem}

The proof adapts the ideas of the proof of~\cite[Lemma~5.9]{R-WWahl}. Here, we just comment on the changes that have to be made to the proof of~\cite[Lemma~5.9]{R-WWahl} in order to prove the above lemma.

\begin{proof}[Outline of the proof.]
    We define $X(M)_{\bullet}$ to be the semisimplicial set with $p$-simplices the split injective $R$-module homomorphisms $f \colon R^{p+1} \rightarrow M$, and with $i$-th face map given by precomposing with the inclusion $R^i \oplus 0 \oplus R^{p - i} \rightarrow R^{p + 1}$. We write $U(M)$ for the simplicial complex with vertices the $R$-module homomorphisms $v \colon R \rightarrow M$ which are split injections (without a~choice of splitting), and where a~tuple $(v_0, \ldots, v_p)$ spans a~$p$-simplex if and only if the sum $v_0 \oplus \ldots \oplus v_p \colon R^{p+1} \rightarrow M$ is a~split injection.

    Note that the poset of simplices of $X(M)_{\bullet}$ is equal to the poset $\calO(M) \cap \calU(M^{\infty})$ and that, given a~$p$-simplex $\sigma = \langle v_0, \ldots, v_p \rangle \in U(M)$, the poset of simplices of the complex $(\Link_{U(M)}(\sigma))_{\bullet}^{ord}$ equals the poset $\calO(M) \cap \calU(M^{\infty})_{(v_0, \ldots, v_p)}$. Hence, by applying Theorem~\ref{ThmGLConnectivity} and arguing as in the proof of~\cite[Lemma~5.9]{R-WWahl} we get that $U(M \oplus R^n)$ is weakly Cohen--Macaulay (as defined in~\cite[Sec.~2.1]{GR-WStabNew}) of dimension $n + \rk(M) - \sr(R)$.

    As in the proof of~\cite[Lemma~5.9]{R-WWahl} we want to show that the assumptions of~\cite[Thm.~3.6]{HW} are satisfied. The complex $S_n(M, R)$ is a~join complex over $U(M \oplus R^n)$ by the same reasoning as in the proof in~\cite{R-WWahl}. In order to show that $\pi(\Link_{S_n(M, R)}(\sigma))$ is weakly Cohen--Macaulay of dimension $n + \rk(M) - \sr(R) - p - 2$ for each $p$-simplex $\sigma \in S_n(M, R)$ we apply Proposition~\ref{PropGLCancellation} instead of~\cite[Prop.~5.8]{R-WWahl} in the proof of~\cite[Lemma~5.9]{R-WWahl}. This shows that the remaining assumptions of~\cite[Thm.~3.6]{HW} are satisfied. Applying this  and \cite[Thm.~2.10]{R-WWahl} then yields the claim.
\end{proof}

Applying Theorems~\cite[Thm.~3.1]{R-WWahl}, \cite[Thm.~3.4]{R-WWahl} and \cite[Thm.~4.20]{R-WWahl} to $(UfR \text{-} \Mod, \oplus, 0)$ yields the following theorem which directly implies Theorem~A.

\begin{thm} \label{ThmMainGL}
    Let $F \colon UfR\text{-}\Mod \rightarrow \Z\text{-}\Mod$ be a~coefficient system of degree~$r$ at~$0$ in the sense of~\cite[Def.~4.10]{R-WWahl}. Then for $s = \rk(M) - \sr(R)$ the map
        \[H_k(\GL(M); F(M)) \rightarrow H_k(\GL(M \oplus R); F(M \oplus R))\]
    is
    \begin{enumerate}
        \item an~epimorphism for $k \leq \frac{s}{2}$ and an~isomorphism for $k \leq \frac{s - 1}{2}$, if $F$ is constant,
        \item an~epimorphism for $k \leq \frac{s - r}{2}$ and an~isomorphism for $k \leq \frac{s - 2 - r}{2}$, if $F$ is split polynomial in the sense of~\cite{R-WWahl},
        \item an~epimorphism for $k \leq \frac{s}{2} - r$ and an~isomorphism for $k \leq \frac{s - 2}{2} - r$.
    \end{enumerate}
    For the commutator subgroup $\GL(M)^{\prime}$ we get that the map
        \[H_k(\GL(M)^{\prime}; F(M)) \rightarrow H_k(\GL(M \oplus R)^{\prime}; F(M \oplus R))\]
    is
    \begin{enumerate}\setcounter{enumi}{3}
        \item an~epimorphism for $k \leq \frac{s - 1}{3}$ and an~isomorphism for $k \leq \frac{s - 3}{3}$, if $F$ is constant,
        \item an~epimorphism for $k \leq \frac{s - 1 - 2r}{3}$ and an~isomorphism for $k \leq \frac{s - 4 - 2r}{3}$, if $F$ is split polynomial in the sense of~\cite{R-WWahl},
        \item an~epimorphism for $k \leq \frac{s - 1}{3} - r$ and an~isomorphism for $k \leq \frac{s - 4}{3} - r$.
    \end{enumerate}
\end{thm}

\section{Homological Stability for Unitary Groups} \label{ChHomStabU}

The aim of this chapter is to prove the analogue of Theorem~\ref{ThmMainGL} for the case of unitary groups of quadratic modules. This again uses the formulation of the standard strategy to prove homological stability by Randal-Williams--Wahl~\cite{R-WWahl}. In this setting we consider the complex of hyperbolic unimodular sequences in a~quadratic module~$M$. For the special case where $M$ is a~hyperbolic module this has been considered in~\cite{MvdK} but the general case requires new ideas. We prove its high connectivity and deduce the assumptions for the machinery of Randal-Williams--Wahl.

\subsection{The Complex and its Connectivity} \label{SecCpxU}

Following \cite{BakQuadForms} and \cite{BakKTheory} let $R$ be a~ring with an~anti-involution $\overline{\phantom{r}}\colon R \rightarrow R$, i.e.\ $\overline {\overline r} = r$ and $\overline{rs} = \overline s \ \overline r$. Fix a~unit~$\eps \in R$ which is a~central element of~$R$ and satisfies $\overline \eps = \eps^{-1}$. Consider a~subgroup~$\Lambda$ of $(R, +)$ satisfying
    \[\Lambda_{\min}:=\{r - \eps \overline r \ | \ r \in R \} \subseteq \Lambda \subseteq \{r \in R \ | \ \eps \overline r = -r\}=:\Lambda_{\max}\]
and $\overline r \Lambda r \subseteq \Lambda$. An~$(\eps, \Lambda)$-\textit{quadratic module} is a~triple $(M, \lambda, \mu)$, where $M$ is a~right $R$-module, $\lambda \colon M \times M \rightarrow R$ is a~sesquilinear form (i.e.\ $\lambda$ is $R$-antilinear in the first variable and $R$-linear in the second), and $\mu \colon M \rightarrow R/\Lambda$ is a~function, satisfying
\begin{enumerate}
    \item $\lambda(x, y) = \eps \overline {\lambda (y, x)}$,
    \item $\mu(x \cdot a) = \overline a \mu(x) a$ for $a \in R$,
    \item $\mu(x+y) - \mu(x) - \mu(y) = \lambda(x, y) \mod \Lambda$.
\end{enumerate}
The direct sum of two quadratic modules $(M_1, \lambda_1, \mu_1)$ and $(M_2, \lambda_2, \mu_2)$ is given by the quadratic module $(M_1 \oplus M_2, \lambda_1 \oplus \lambda_2, \mu_1 \oplus \mu_2)$, where
\begin{align*}
    (\lambda_1 \oplus \lambda_2) ((m_1, m_2),(m_1^{\prime}, m_2^{\prime})) & := \lambda_1(m_1, m_1^{\prime}) + \lambda_2(m_2, m_2^{\prime}), \\
    (\mu_1 \oplus \mu_2) (m_1, m_2) & := \mu_1(m_1) + \mu_2(m_2),
\end{align*}
for $m_i, m_i^{\prime} \in M_i$. The \textit{unitary group} is defined as
    \[U(M) := \{A \in \GL(M) \ | \ \lambda(Ax, Ay) = \lambda(x, y), \mu(Ax) = \mu(x) \text{ for all } x, y \in M\}.\]
The \textit{hyperbolic module}~$H$ over~$R$ is the $(\eps, \Lambda)$-quadratic module given by
    \[\left( R^2 \text{ with basis } e, f; \begin{pmatrix} 0 & 1\\ \eps & 0 \end{pmatrix}; \mu(e) = \mu(f) = 0 \right). \]
We write~$H^g$ for the direct sum of~$g$ copies of the hyperbolic module~$H$.

Examples of unitary groups for the quadratic module~$H^g$ with various choices of $(R, \eps, \Lambda)$ can be found in~\cite[Ex.~6.1]{MvdK}.

\begin{defn}
A~ring~$R$ satisfies the transitivity condition~$(\mathrm{T}_n)$ if the groups $EU^{\eps}(H^n, \Lambda)$, which is the subgroup of~$U(H^n)$ consisting of elementary matrices as defined in~\cite[Ch.~6]{MvdK}, acts transitively on the set
    \[C^{\eps}_{r}(R, \Lambda) := \{ x \in R^{2n} \ | \ x \text{ is unimodular, } \mu(x) = r \mod \Lambda \}\]
for every $r \in R$. The ring~$R$ has \textit{unitary stable range}~$(\mathrm{US}_n)$ if it satisfies the stable range condition~$(\mathrm{S}_n)$, as defined in Definition~\ref{DefStableRank}, as well as the transitivity condition~$(\mathrm{T}_{n+1})$. We say that~$R$ has \textit{unitary stable rank}~$n$, $\usr(R) = n$, if~$n$ is the least number such that $(\mathrm{US}_n)$ holds and $\usr(R) = \infty$ if such an~$n$ does not exist.
\end{defn}

The transitivity condition~$(\mathrm{T}_n)$ and hence the unitary stable range~$(\mathrm{US}_n)$ are conditions on the triple~$(R, \eps, \Lambda)$ and not just on~$R$. However, to make our notation consistent with the literature we write~$\usr(R)$ as introduced above which drops both $\eps$ and $\Lambda$.

As remarked in~\cite[Rmk.~6.4]{MvdK} we have $\usr(R) \leq \asr(R) + 1$ for the absolute stable rank of Magurn--van der Kallen--Vaserstein~\cite{MvdKV}. In the special case where the involution on~$R$ is the identity map (which implies that $R$ is commutative), we have $\usr(R) \leq \asr(R)$. We now give some well-known examples of rings and their unitary stable rank.

\begin{exs} \label{ExsUsr}
    The following examples work for any anti-involution on~$R$ and every choice of $\eps$ and $\Lambda$.
    \begin{enumerate}
        \item Let $R$ be a~commutative Noetherian ring of finite Krull dimension~$d$. Then any $R$-algebra~$A$ that is finitely generated as an~$R$-module satisfies $\usr(A) \leq d + 2$. \cite[Thm.~3.1]{MvdKV}
        \item A~semi-local ring satisfies $\usr(R) \leq 2$. \cite[Thm.~2.4]{MvdKV}
        \item \label{ExsUsrVirtuallyPolycyclic} For a~virtually polycyclic group~$G$ we have $\usr(\Z[G]) \leq h(G) + 3$, where $h(G)$ is the Hirsch length as defined in Example~\ref{ExsSr}~(4). \cite[Thm.~7.3]{CrowleySixt}
    \end{enumerate}
\end{exs}

A~sequence $(v_1, \ldots, v_k)$ of elements in the quadratic module $(M, \lambda, \mu)$ is called \textit{unimodular} if the sequence is unimodular in~$M$ considered as an~$R$-module (see Section~\ref{SecCpxGL}). We say that the sequence is \textit{$\lambda$-unimodular} if there are elements $w_1, \ldots, w_k$ in~$M$ such that $\lambda(w_i, v_j) = \delta_{i, j}$, where $\delta_{i, j}$ denotes the Kronecker delta. We write $\calU(M)$ and $\calU(M, \lambda)$ for the subposet of unimodular and $\lambda$-unimodular sequences in~$M$ respectively.

Note that every $\lambda$-unimodular sequence is in particular unimodular. The following lemma shows that there are cases where the converse is also true.

\begin{lem} \label{LemNonSingular}
    Let the sequence $(v_1, ..., v_k)$ be unimodular in~$M$. If there is a~submodule $N \subseteq M$ containing the $v_i$ such that $\lambda\vert_N$ is non-singular, then the sequence $(v_1, ..., v_k)$ is $\lambda$-unimodular in~$N$.
\end{lem}

\begin{proof}
    Let $(v_1, \ldots, v_k)$ be a~unimodular sequence in $M$. This means that there are maps $f_1, \ldots, f_k \colon R \rightarrow M$ with $f_i(1) = v_i$ and maps $\phi_1, \ldots, \phi_k \colon M \rightarrow R$ with $\phi_j \circ f_i = \delta_{i, j} \cdot \mathds{1}_R$. Note that this implies that $\phi_j(v_i) = \delta_{i, j}$. Now, $\lambda$ being non-singular on~$N$ means that the map
    \begin{eqnarray*}
        N
        & \longrightarrow
        & N^* \\
        v
        & \longmapsto
        & \lambda(-, v)
    \end{eqnarray*}
    is an~isomorphism. Hence, there are $w^{\prime}_1, \ldots, w^{\prime}_k \in N$ such that $\lambda(-, w^{\prime}_i) = \phi_i(-)$ on~$N$. Defining $w_i := w^{\prime}_i \eps$ then yields
        \[\lambda(w_i, v_j) = \lambda(w^{\prime}_i \eps, v_j) = \eps \overline{\lambda(v_j, w^{\prime}_i) \eps} = \eps \overline{\eps} \overline{\phi_i(v_j)} = \delta_{i, j}.\qedhere\]
\end{proof}

We call a~subset~$S$ of a~quadratic module $(M, \lambda, \mu)$ \textit{isotropic} if $\mu(x) = 0$ and $\lambda(x, y) = 0$ for all $x, y \in S$. Let $\calI \calU(M)$ denote the set of $\lambda$-unimodular sequences $(x_1, \ldots, x_k)$ in $M$ such that $x_1, \ldots, x_k$ span an~isotropic direct summand of~$M$. We write $\calH \calU(M)$ for the set of sequences $((x_1, y_1), \ldots, (x_k, y_k))$ such that $(x_1, \ldots, x_k)$, $(y_1, \ldots, y_k) \in \calI \calU(M)$, and $\lambda(x_i, y_j) = \delta_{i, j}$. This can also be thought of as the set of quadratic module maps $H^k \rightarrow M$. We call $\calI\calU(M)$ the \textit{poset of isotropic $\lambda$-unimodular sequences} and $\calH\calU(M)$ the \textit{poset of hyperbolic $\lambda$-unimodular sequences}. We say that the sequence $x = ((x_1, y_1), \ldots, (x_k, y_k)) \in \calH \calU(M)$ is of length~$|x| = k$.

Let $\calM \calU(M)$ be the set of sequences $((x_1, y_1), \ldots, (x_k, y_k)) \in \calO(M \times M)$ satisfying
\begin{enumerate}
    \item $(x_1, \ldots, x_k) \in \calI \calU(M)$,
    \item for each $i$ we have either $y_i = 0$ or $\lambda(x_j, y_i) = \delta_{j, i}$,
    \item the span $\langle y_1, \ldots, y_k \rangle$ is isotropic.
\end{enumerate}
We identify the poset $\calI \calU(M)$ with $\calM \calU(M) \cap \calO(M \times \{0\})$ and the poset $\calH \calU(M)$ with $\calM \calU(M) \cap \calO(M \times (M \setminus \{0\}))$.

In order to phrase the main theorem of this section we introduce the following notion: For an~$(\eps, \Lambda)$-quadratic module $(M, \lambda, \mu)$ define the \textit{Witt index} as
    \[g(M) := \sup\{g \in \N \ | \ \text{there is a~quadratic module $P$ such that } M \cong P \oplus H^g \}.\]

\begin{thm} \label{Thm7.4MvdK}
    The poset $\calH \calU(M)$ is $\bigl \lfloor \frac{g(M) - \usr(R) - 3}{2} \bigr \rfloor$-connected and for every $x \in \calH \calU(M)$ the poset $\calH \calU(M)_x$ is $\bigl \lfloor \frac{g(M) - \usr(R) - |x| - 3}{2} \bigr \rfloor$-connected.
\end{thm}

For the special case where the quadratic module $M$ is a~direct sum of hyperbolic modules~$H^n$, Theorem~\ref{Thm7.4MvdK} has been proven by Mirzaii--van der Kallen in~\cite[Thm.~7.4]{MvdK}. Galatius--Randal-Williams have treated the case of general quadratic modules over the integers.

In order to prove Theorem~\ref{Thm7.4MvdK} we need the following lemma which extends \cite[Lemma~6.6]{MvdK} to the case of general quadratic modules. Note, however, that the proof is not an~extension of the proof of~\cite[Lemma~6.6]{MvdK} but rather uses techniques of Vaserstein~\cite{VasersteinDim}. A~similar statement has been given by Petrov in~\cite[Prop.~6]{Petrov}. However, Petrov considers hyperbolic modules which are defined over rings with a~pseudoinvolution and only allows $\eps = -\overline 1$. He also states his connectivity range using a~different rank, called the \textit{$\Lambda$-stable rank}, which we shall not discuss.

\begin{lem} \label{Lem6.6MvdK}
    Let $P \oplus H^g$ be a~quadratic module. If $g \geq \usr(R) + k$ and $(v_1, \ldots, v_k) \in \calU(P \oplus H^g, \lambda)$ then there is an~automorphism $\phi \in U(P \oplus H^g)$ such that $\phi(v_1, \ldots, v_k) \subseteq P \oplus H^k$ and the projection of $\phi(v_1, \ldots, v_k)$ to the hyperbolic~$H^k$ is $\lambda$-unimodular.
\end{lem}

The following section contains the necessary foundations as well as the proof of Lemma~\ref{Lem6.6MvdK}.

\subsection{Proof of Lemma~\ref{Lem6.6MvdK}} \label{SecProofLem6.6}

Following~\cite{VasersteinDim} an~$(n+k) \times k$-matrix~$A$ is called \textit{unimodular} if it has a~left inverse. Note that the matrix~$A$ is unimodular if and only if the matrix~$CA$ is unimodular for any invertible matrix~$C \in \GL_{n+k}(R)$. A~ring~$R$ is said to satisfy the condition~$(\mathrm{S}^k_n)$ if for every unimodular $(n + k) \times k$-matrix~$A$, there exists an~element~$r \in R^{n + k - 1}$ such that
\[\left(\begin{tabular}{c|c}
    $\mathds{1}_{n + k - 1}$ & $r^{\top}$ \\
    \hline
    $0$&$1$
\end{tabular}\right)
\cdot A =
\begin{pmatrix} B \\ u \end{pmatrix},\]
where the matrix~$B$ is unimodular and $u$ is the last row of~$A$.

Note that condition~$(\mathrm{S}^1_n)$ is the same as condition~$(S_n)$. Furthermore, Vaserstein shows in~\cite[Thm.~3\,$^\prime$]{VasersteinDim} shows that the condition~$(\mathrm{S}^k_n)$ is equivalent to the condition~$(\mathrm{S}_n)$.

\subsubsection{$n \times k$-Blocks}

Given a~quadratic $R$-module~$M$ we define an~\textit{$n \times k$-block~$A$ for~$M$} to be an~$n \times k$-matrix~$\left(r_{i, j}\right)_{i, j}$ with entries in~$R$ together with $k$~anti-linear maps $f_1, \ldots, f_k \colon M \rightarrow R$. We will write this data as
\[A = \begin{pmatrix}
    r_{1,1} & \ldots & r_{1,k} \\
    \vdots & & \vdots \\
    r_{n,1} & \ldots & r_{n,k} \\
    f_1 & \ldots & f_k
\end{pmatrix}.\]
Note that with this notation an~$n \times k$-block has in fact $n+1$ rows. We refer to the row of maps $(f_1, \ldots, f_k)$ as the \textit{last row} of~$A$. Given an~$(n+1) \times (n+1)$-matrix of the form
\[\left(\begin{tabular}{ccc|c}
    $s_{1,1}$ & $\ldots$ & $s_{1,n}$ & $m_1$ \\
    $\vdots$ &  & $\vdots$ & $\vdots$ \\
    $s_{n,1}$ & $\ldots$ & $s_{n,n}$ & $m_{2g}$ \\
    \hline
    $0$ & $\ldots$ & $0$ & $s$
\end{tabular}\right),\]
where $s, s_{i,j} \in R$, $m_i \in M$, we can act with it from the left on an~$n \times k$-block~$A$ by matrix multiplication, where we define
    \[m_i \cdot f_j := f_j(m_i) \text{ and } s \cdot f_j := f_j( - \cdot \overline s).\]
We can act from the right on the block~$A$ with a~$k \times k$-matrix with entries in~$R$ again by matrix multiplication, where we define $f_j \cdot r$ to send an~element~$m \in M$ to $f_j(m) \cdot r$ for $r \in R$.

\begin{defn}
    We say that an~$n \times k$-block $A$ is \textit{unimodular} if there is a~$k \times (n+1)$-matrix~$A_L$ of the form
    \[\begin{pmatrix}
        r^{\prime}_{1,1} & \ldots & r^{\prime}_{1,n} & m^{\prime}_1 \\
        \vdots &  & \vdots & \vdots \\
        r^{\prime}_{k,1} & \ldots & r^{\prime}_{k,n} & m^{\prime}_k
    \end{pmatrix}\]
    with $r^{\prime}_{i, j} \in R$ and $m^{\prime}_i \in M$, such that $A_L \cdot A = \mathds{1}_k$, where the multiplication is again given by matrix multiplication, with $m^{\prime}_i \cdot f_j$ as defined above.
\end{defn}

Note that the $n \times k$-block~$A$ is unimodular if and only if one of the following blocks is unimodular:
\[\left(\begin{tabular}{c|c}
        $1$ & $0$ \\
        \hline
        $0$ & \multirow{4}{*}{$A$} \\
        $\vdots$ & \\
        $0$ & \\
        $f$ &
    \end{tabular}\right), \
    \left(\begin{tabular}{c|c}
            $\mathds{1}_n$ & $v^{\top}$ \\
            \hline
            $0$ & $1$ \\
        \end{tabular}\right) \cdot A, \
    \left(\begin{tabular}{c|c}
            $C$ & $0$ \\
            \hline
            $0$ & $1$ \\
    \end{tabular}\right) \cdot A, \ \text{ or }
    A \cdot
    \left(\begin{tabular}{c|c}
        $1$ & $v$ \\
        \hline
        $0$ & $\mathds{1}_n$ \\
    \end{tabular}\right),\]
    for a~vector~$v \in R^n$ and a~matrix~$C \in \GL_n(R)$.

\begin{defn}
    An~$n \times k$-block~$A$ for~$M$ is \textit{matrix reducible} if there is a~vector $m \in M^n$ such that
    \[\left(\begin{tabular}{c|c}
        $\mathds{1}_n$ & $m^{\top}$ \\
        \hline
        $0$ & $1$ \\
    \end{tabular}\right)
    \cdot A =
    \begin{pmatrix} B \\ u \end{pmatrix},\]
    where the $n \times k$-matrix~$B$ is unimodular and $u$ is the last row of the block~$A$.
\end{defn}

\begin{prop}\label{PropSnkForBlocks}
    If $k + \sr(R) \leq n + 1$ then every unimodular $n \times k$-block~$A$ is matrix reducible.
\end{prop}

Matrix reducibility is preserved under certain operations as the following proposition shows (cf.\ proof of~\cite[Thm.~3\,$^\prime$]{VasersteinDim}).

\begin{prop} \label{PropMoves}
    Let $A$ be an~$n \times k$-block for $M$. Then $A$ is matrix reducible if and only if the block obtained from~$A$ by doing any of the following moves is matrix reducible.
    \begin{enumerate}
        \item Multiply on the left by a~matrix of the form
        \[\left(\begin{tabular}{c|c}
            $\mathds{1}_n$ & $v^{\top}$ \\
            \hline
            $0$&$1$
        \end{tabular}\right),\]
        for an~element $v \in M^n$.
        \item Multiply on the left by a~matrix of the form
        \[\left(\begin{tabular}{c|c}
            $C$ & $0$ \\
            \hline
            $0$&$1$
        \end{tabular}\right),\]
        for a~matrix $C \in \GL_n(R)$.
        \item Multiply on the right by a~matrix $D \in \GL_k(R)$.
    \end{enumerate}
\end{prop}

\begin{proof}
    Note that each of the above moves may be inverted by a~move of the same type. It is therefore enough to show that if $A$ is matrix reducible then so is the block obtained from~$A$ by doing one of the above moves. Let $m \in M^n$ be the sequence showing that the block~$A$ is matrix reducible, i.e.\ we have
    \[\left(\begin{tabular}{c|c}
        $\mathds{1}_n$ & $m^{\top}$ \\
        \hline
        $0$ & $1$ \\
    \end{tabular}\right)
    \cdot A =
    \begin{pmatrix} B \\ u \end{pmatrix},\]
    where the $n \times k$-matrix~$B$ is unimodular.

    Statement (1) follows from the fact that multiplying two of these matrices with last column $(v_1, 1)$ and $(v_2, 1)$ respectively yields another matrix of this form whose last column is given by $(v_1 + v_2, 1)$.

    To show (2) we can write
    \[\left(\begin{tabular}{c|c}
        $C$ & $0$ \\
        \hline
        $0$&$1$
    \end{tabular}\right)
    \cdot A=
    \left[\left(
    \begin{tabular}{c|c}
        $C$ & $0$ \\
        \hline
        $0$&$1$
    \end{tabular}\right)
    \cdot
    \left(\begin{tabular}{c|c}
        $\mathds{1}_n$ & $m^{\top}$ \\
        \hline
        $0$&$1$
    \end{tabular}\right)
    \cdot
    \left(\begin{tabular}{c|c}
        $C^{-1}$ & $0$ \\
        \hline
        $0$&$1$
    \end{tabular}\right)\right]
    \cdot
    \left[\left(\begin{tabular}{c|c}
        $C$ & $0$ \\
        \hline
        $0$&$1$
    \end{tabular}\right)
    \cdot
    \begin{pmatrix} B \\ u \end{pmatrix} \right],\]
    where the product of the first three matrices is
    \[\left(\begin{tabular}{c|c}
        $\mathds{1}_n$ & $C m^{\top}$ \\
        \hline
        $0$&$1$
    \end{tabular}\right)\]
    and the product of the last two matrices is $\begin{pmatrix} CB \\ u \end{pmatrix}$. Note that multiplying a~unimodular matrix by an~invertible matrix on either side yields again a~unimodular matrix. Thus, $-C m^{\top}$ is the corresponding sequence for the block
    \[\left(\begin{tabular}{c|c}
        $C$ & $0$ \\
        \hline
        $0$&$1$
    \end{tabular}\right)
    \cdot A.\]

    For (3) note that multiplying the matrix $\begin{pmatrix} B \\ u \end{pmatrix}$ on the right by~$D$ yields a~matrix~$\begin{pmatrix} BD \\ u^{\prime} \end{pmatrix}$. As noted in part~(2), the matrix $BD$ is also unimodular so $m$ is also the sequence to show that the block $AD$ is matrix reducible.
\end{proof}

\begin{proof}[Proof of Proposition~\ref{PropSnkForBlocks}]
    Let us write the unimodular $n \times k$-block as
    \[A =
    \begin{pmatrix}
        r_{1,1} & \ldots & r_{1,k} \\
        \vdots & & \vdots \\
        r_{n,1} & \ldots & r_{n,k} \\
        f_1 & \ldots & f_k
    \end{pmatrix}.\]
    The proof is by induction on~$k$.

    Let $k=1$. Since the block~$A$ is unimodular, there is a~left inverse $A_L := ((r^{\prime}_1)^{\top}, \ldots, (r^{\prime}_n)^{\top}, (m^{\prime})^{\top})$ of~$A$ for vectors $r^{\prime}_i \in R^k$ and $m^{\prime} \in M^k$. Hence, the sequence $(r_{1,1}, \ldots, r_{n,1}, f_1(m^{\prime}_1)) \in R^{n+1}$ is unimodular by construction and since $n + 1 > \sr(R)$ there are $v_1, \ldots, v_n \in R$ such that the sequence
        \[(r_{1,1} + v_1 f_1(m^{\prime}_1), \ldots, r_{n,1} + v_n f_1(m^{\prime}_1))\]
    is unimodular. Defining $m_i := m^{\prime} \cdot \overline v_i$ then yields the base case.

    Let us assume that the statement is true for $k-1$ and consider the case $k > 1$. Since $A$ is a~unimodular block, in particular the first column~$(r_1)^{\top}$ is unimodular having a~left inverse $(r^{\prime}_{1,1}, \ldots, r^{\prime}_{1, n}, m^{\prime}_1)$ which is the first row of the left inverse~$A_L$ of~$A$. Hence, the sequence $(r_{1,1}, \ldots, r_{n,1}, f_1(m^{\prime}_1))$ is unimodular. By assumption we have $n+1 > \sr(R)$, so there is a~vector $v := (v_1, \ldots, v_n) \in R^n$ such that the sequence
        \[ r_1^{\prime} := r_{1,1} + v_1 f_1(m^{\prime}_1), \ldots, r_{n,1} + v_n f_1(m^{\prime}_1) \in R^n\]
    is unimodular. Thus, there is an~$C \in \GL_n(R)$ such that $C r_1^{\prime} = (1, 0, \ldots, 0)$. Consider the block
    \[A_1 :=
    \left(\begin{tabular}{c|c}
        $C$ & $0$ \\
        \hline
        $0$ & $1$
    \end{tabular}\right)
    \cdot
    \left(\begin{tabular}{c|c}
        $\mathds{1}_n$ & $v^{\top}$ \\
        \hline
        $0$ & $1$
    \end{tabular}\right)
    \cdot A.\]
    Then $A_1$ is of the form
    \[\left(\begin{tabular}{c|c}
        $1$ & $u^{\prime}$ \\
        \hline
        $0$ & \multirow{4}{*}{$A^{\prime}$} \\
        $\vdots$ &  \\
        $0$ &  \\
        $f_1$ &
    \end{tabular} \right)\]
    for an~$(n-1) \times (k-1)$-block~$A^{\prime}$ for~$M$. Now, by Proposition~\ref{PropMoves} the block~$A$ is matrix reducible if and only if the block~$A_1$ is matrix reducible. Proposition~\ref{PropMoves} also implies that this is equivalent to the block
    \[A_2 := A_1 \cdot
    \left(\begin{tabular}{c|c}
        $1$ & $-u^{\prime}$ \\
        \hline
        $0$ & $\mathds{1}_n$
    \end{tabular}\right)
    =
    \left(\begin{tabular}{c|c}
        $1$ & $0$ \\
        \hline
        $0$ & \multirow{4}{*}{$A^{\prime \prime}$} \\
        $\vdots$ & \\
        $0$ & \\
        $f_1$ &
    \end{tabular}\right)\]
    being matrix reducible. Therefore, it is enough to show that the block~$A_2$ is matrix reducible. Since the block~$A$ is unimodular, so is $A_2$ as remarked above. This implies that the block~$A^{\prime \prime}$ is unimodular as well. Hence, by the induction hypothesis there is a~vector $m \in M^{n-1}$ such that
    \[\left(\begin{tabular}{c|c}
        $1$ & $-u^{\prime}$ \\
        \hline
        $0$ & $\mathds{1}_n$
    \end{tabular}\right)
    \cdot A^{\prime \prime} =
    \begin{pmatrix}
        \tilde B \\
        \tilde u
    \end{pmatrix},\]
    where the matrix~$\tilde B$ is unimodular and $\tilde u$ is the last row of~$A^{\prime \prime}$. Thus,
    \[\left(\begin{tabular}{c|c|c}
        $1$ & $0$ & $0$ \\
        \hline
        \multirow{2}{*}{$0$} & \multirow{2}{*}{$\mathds{1}_{n-1}$} & \multirow{2}{*}{$m^{\top}$} \\
        & & \\
        \hline
        $0$ & $0$ & $1$
    \end{tabular} \right)
    \cdot A_2 =
    \begin{pmatrix}
        1 & 0 \\
        * & \tilde B \\
        * & \tilde u
    \end{pmatrix},\]
    where the matrix $\begin{pmatrix} 1 & 0 \\ * & \tilde B \end{pmatrix}$ is unimodular since $\tilde B$ is unimodular.
\end{proof}

The next proposition is an~extension of~\cite[Thm.~1]{VasersteinDim}.

\begin{prop}\label{PropThm1SnkVas}
    Let $k + \sr(R) = n + 1$ and $l>0$ then for any unimodular $(n+l)\times k$-block~$A$ there is a~vector $m \in M^n$ such that
    \[\left(\begin{tabular}{c|c|c}
        \multirow{2}{*}{$\mathds{1}_n$} & \multirow{2}{*}{$0$} & \multirow{2}{*}{$m^{\top}$} \\
        & & \\
        \hline
        \multirow{2}{*}{$0$} & \multirow{2}{*}{$\mathds{1}_l$} & \multirow{2}{*}{$0$} \\
        & & \\
        \hline
        $0$ & $0$ & $1$
    \end{tabular}\right)
    \cdot A = \begin{pmatrix} B \\ u \end{pmatrix},\]
    where the $(n + l) \times k$-matrix~$B$ is unimodular and $u$ is the last row of the block~$A$.
\end{prop}

\begin{proof}
    Since $A$ is a~unimodular $(n+l)\times k$-block by Proposition~\ref{PropSnkForBlocks} there is an~element $\tilde m \in M^{n+l}$ such that
    \[\left(\begin{tabular}{c|c}
        $\mathds{1}_{n+l}$ & $\tilde m^{\top}$ \\
        \hline
        $0$ & $1$
    \end{tabular}\right)
    \cdot A = \begin{pmatrix} B_1 \\ u_1 \end{pmatrix},\]
    where the $(n+l) \times k$-matrix~$B_1$ is unimodular and $u_1 = u$ is the last row of the block~$A$. Since $l>0$ and $n + l - k \geq \sr(R)$ we can now apply the condition $(\mathrm{S}^k_{n+l-k})$ to the unimodular matrix~$B_1$ to get an~element $v \in R^{n+l-1}$ such that
    \[\left(\begin{tabular}{c|c}
        $\mathds{1}_{n+l-1}$ & $v^{\top}$ \\
        \hline
        $0$ & $1$\\
    \end{tabular}\right)
    \cdot B_1 = \begin{pmatrix} B_2 \\ u_2 \end{pmatrix},\]
    where the $(n+l-1)\times k$-matrix~$B_2$ is unimodular and $u_2$ is the last row of the matrix~$B_1$. Together we get
    \[\left(\begin{tabular}{c|c|c}
        \multirow{2}{*}{$\mathds{1}_{n+l-1}$} & \multirow{2}{*}{$v^{\top}$} & \multirow{2}{*}{$0$} \\
        & & \\
        \hline
        $0$ & $1$ & $0$ \\
        \hline
        $0$ & $0$ & $1$
    \end{tabular}\right)
    \cdot
    \left(
    \begin{tabular}{c|c}
        $\mathds{1}_{n+l}$ & $\tilde m^{\top}$ \\
        \hline
        $0$ & $1$
    \end{tabular}\right)
    \cdot A =
    \begin{pmatrix} B_2 \\ u_2 \\ u_1 \end{pmatrix}.\]
    Notice that the product of the first two matrices can be written in the form
    \[\left(\begin{tabular}{c|c|c}
        \multirow{2}{*}{$\mathds{1}_{n+l-1}$} & \multirow{2}{*}{$*$} & \multirow{2}{*}{$*$} \\
         & & \\
        \hline
        $0$ & $1$ & $*$ \\
        \hline
        $0$ & $0$ & $1$ \\
    \end{tabular}\right),\]
    where the last column has entries in the module~$M$ and the rest of the matrix has entries in the ring~$R$. Iterating this yields a~matrix
    \[C :=
    \left(\begin{tabular}{c|ccc|c}
        \multirow{2}{*}{$\mathds{1}_n$} & \multicolumn{3}{c|}{\multirow{2}{*}{$*$}} & \multirow{2}{*}{$*$} \\
         & \multicolumn{3}{c|}{} & \\
        \hline
        \multirow{3}{*}{\vspace{-10px}$0$} & $1$ & $*$ & $*$ & \multirow{3}{*}{\vspace{-10px}$*$} \\
        & $0$ & $\ddots$ & $*$ & \\
        & $0$ & $0$ & $1$ & \\
        \hline
        $0$ & \multicolumn{3}{c|}{$0$} & $1$ \\
    \end{tabular}\right)\]
    and $C \cdot A$ is a~matrix of the form $\begin{pmatrix} B^{\prime} \\ B^{\prime \prime} \end{pmatrix}$, where $B^{\prime}$ is an~$n \times k$-matrix and $B^{\prime \prime}$ is an~$(l + 1) \times k$-block. The matrix $B^{\prime}$ is unimodular by construction. Note that row operations involving only the rows of~$B^{\prime \prime}$ do not change the matrix~$B^{\prime}$. Hence, we can change the above matrix~$C$ to be of the form
    \[C^{\prime} :=
    \left(\begin{tabular}{c|c|c}
        \multirow{2}{*}{$\mathds{1}_n$} & \multirow{2}{*}{$*$} & \multirow{2}{*}{$*$} \\
         & & \\
        \hline
        \multirow{2}{*}{\vspace{-10px}$0$} & \multirow{2}{*}{$\mathds{1}_l$} & \multirow{2}{*}{$0$} \\
        & & \\
        \hline
        $0$ & $0$ & $1$ \\
    \end{tabular}\right).\]
    Again, $C^{\prime} \cdot A$ is a~matrix of the form $\begin{pmatrix} B^{\prime} \\ \tilde B^{\prime \prime} \end{pmatrix}$, where $B^{\prime}$ is the same matrix as above and hence unimodular. Instead of dividing this matrix into the first $n$ and the last $l+1$ rows, let us now divide it into the first $n+l$ and the last row, written as $\begin{pmatrix} B^{\prime \prime \prime} \\ u \end{pmatrix}$, where $u$ is by construction the last row of the matrix~$A$. Since the matrix~$B^{\prime}$ is unimodular, so is the matrix~$B^{\prime \prime \prime}$. Row operations on~$B^{\prime \prime \prime}$ correspond to multiplying $B^{\prime \prime \prime}$ on the left by invertible matrices which keeps the matrix unimodular. Hence, we can perform row operations on~$C^{\prime}$ using all but the last row to get a~matrix of the form
    \[\left(\begin{tabular}{c|c|c}
        \multirow{2}{*}{$\mathds{1}_n$} & \multirow{2}{*}{$0$} & \multirow{2}{*}{$m^{\top}$} \\
        & & \\
        \hline
        \multirow{2}{*}{\vspace{-10px}$0$} & \multirow{2}{*}{$\mathds{1}_l$} & \multirow{2}{*}{$0$} \\
        & & \\
        \hline
        $0$ & $0$ & $1$ \\
    \end{tabular}\right).\]
    This finishes the proof.
\end{proof}

We immediately get the following corollary.

\begin{cor}\label{CorOfPropThm1SnkVas}
    Let $k + \sr(R) = n + 1$ and $l>0$ then for any unimodular $(n+l)\times k$-block~$A$ there is a~vector $m \in M^n$ and an~$n \times l$-matrix~$Q$ with entries in $R$ such that
    \[\left(\begin{tabular}{c|c|c}
        \multirow{2}{*}{$\mathds{1}_n$} & \multirow{2}{*}{$Q$} & \multirow{2}{*}{$m^{\top}$} \\
        & & \\
        \hline
        \multirow{2}{*}{$0$} & \multirow{2}{*}{$\mathds{1}_l$} & \multirow{2}{*}{$0$} \\
        & & \\
        \hline
        $0$ & $0$ & $1$
    \end{tabular}\right)
    \cdot A = \begin{pmatrix} B_1 \\ B_2 \\ u \end{pmatrix},\]
    where the $n \times k$-matrix~$B_1$ is unimodular and $\begin{pmatrix} B_2 \\ u \end{pmatrix}$ are the last $l + 1$ rows of the block~$A$.
\end{cor}

\begin{proof}
    The matrix $C^{\prime}$ constructed in the proof of Proposition~\ref{PropThm1SnkVas} is the required matrix.
\end{proof}

\subsubsection{Orthogonal Transvections}

Following~\cite[Ch.~7]{MvdKV} let $e$ and $u$ be elements in the quadratic module $(M, \lambda, \mu)$ satisfying $\mu(e) = 0$ and $\lambda(e, u) = 0$. For $x \in \mu(u)$ we define an~automorphism $\tau(e, u, x)$  of  the quadratic module~$M$ by
    \[\tau(e, u, x)(v) = v + u\lambda(e, v) - e \overline{\eps} \lambda(u, v) - e \overline{\eps}x \lambda(e, v).\]
If $e$ is $\lambda$-unimodular, the map $\tau(e, u, x)$ is called an~\textit{orthogonal transvection}.

The following is the last ingredient in order to prove Lemma~\ref{Lem6.6MvdK}.

\begin{prop} \textup(\cite[Prop.~5.12]{R-WWahl}\textup) \label{PropCancellationH}
    Let $M$ be a~quadratic module and $M \oplus H \cong H^{g + 1}$. If $g \geq \usr(R)$ then $M \cong H^g$.
\end{prop}

\begin{proof}[Proof of Lemma~\ref{Lem6.6MvdK}]
    In the following we adapt the ideas of~Step~$1$ in the proof of~\cite[Thm.~8.1]{MvdKV}. Let $(v_1, \ldots, v_k)$ be a~$\lambda$-unimodular sequence in the quadratic module $P \oplus H^g$ with $g \geq \usr(R) + k$. Recall that we want to show that there is an~automorphism $\phi \in U(P \oplus H^g)$ such that $\phi(v_1, \ldots, v_k) \subseteq P \oplus H^k$ and the projection of $\phi(v_1, \ldots, v_k)$ to the hyperbolic~$H^k$ is $\lambda$-unimodular. Denoting the basis of~$H^g$ by $e_1, f_1, \ldots, e_g, f_g$ we can write
        \[v_i = p_i + \sum_{l = 1}^g e_l A^i_l + \sum_{l = 1}^g f_l B^i_l \quad \text{ for } p_i \in P \text{ and } A^i_l, B^i_l \in R.\]
    As the sequence $(v_1, \ldots, v_k)$ is $\lambda$-unimodular, there are
        \[w_i = q_i + \sum_{l = 1}^g e_l a^i_l + \sum_{l = 1}^g f_l b^i_l \quad \text{ for } q_i \in P \text{ and } a^i_l, b^i_l \in R\]
    satisfying
    \begin{eqnarray*}
        \delta_{i,j}
        & =
        & \lambda(w_i, v_j) = (q_i, a^i_1, b^i_1, \ldots, b^i_g)
        \left(\begin{tabular}{c|cc|c}
            $\lambda|_P$ & $0$ & $0$ & $\cdots$ \\
            \hline
            $0$ & $0$ & $1$ & \multirow{2}{*}{$\cdots$} \\
            $0$ & $\eps$ & $0$ & \\
            \hline
            $\vdots$ & \multicolumn{2}{c|}{$\vdots$} & $\ddots$
        \end{tabular}\right)
    \begin{pmatrix}
        p_j \\
        A^j_1 \\
        B^j_1 \\
        \vdots \\
        B^j_g
    \end{pmatrix} \\
        & =
    &   \lambda(q_i, p_j) + \sum_{l = 1}^g a^i_l B^j_l + \eps \sum_{l = 1}^g b^i_l A^j_l
    \end{eqnarray*}
    Note that a~sequence $(v_1\, \ldots, v_k)$ is $\lambda$-unimodular if and only if its associated block
    \[A_{(v_1, \ldots, v_k)} := \begin{pmatrix}
        A^1_1           & \ldots & A^k_1 \\
        B^1_1           & \ldots & B^k_1 \\
        \vdots          &        & \vdots \\
        B^1_g           & \ldots & B^k_g \\
        \lambda(-, p_1) & \ldots & \lambda(-, p_k)
    \end{pmatrix}\]
    is unimodular. Since $g - k + 1 > \sr(R)$ by Proposition~\ref{PropThm1SnkVas} there are $\tilde p_1, \ldots, \tilde p_g \in P$ such that
    \[\left(\begin{tabular}{cc|c}
        \multicolumn{2}{c|}{\multirow{5}{*}{$\mathds{1}_{2g}$}} & $\tilde p_1$ \\
        \multicolumn{2}{c|}{} & $0$ \\
        \multicolumn{2}{c|}{} & $\vdots$ \\
        \multicolumn{2}{c|}{} & $\tilde p_g$ \\
        \multicolumn{2}{c|}{} & $0$ \\
        \hline
        \multicolumn{2}{c|}{$0$} & $1$
    \end{tabular}\right)
    \cdot A_{(v_1, \ldots, v_k)}=
    \begin{pmatrix} B \\ u \end{pmatrix},\]
    where the matrix~$B$ is unimodular. Strictly speaking this is not of the form of Proposition~\ref{PropThm1SnkVas} but we can reorder the basis of the matrix part to get the above statement. Now, for $y_i \in \mu(\tilde p_i)$ consider the following composition of transvections
        \[\tilde {\phi} :=  \tau(e_g, - \overline{\eps} \tilde p_g, \eps y_g \overline{\eps}) \circ \ldots \circ \tau(e_1, - \overline{\eps} \tilde p_1, \eps y_1 \overline{\eps}).\]
    Then by induction we have
        \[\tilde{\phi}(v_i) = v_i + \sum_{j = 1}^g\left( - \overline{\eps} \tilde p_j B_j^i + e_j \lambda(\tilde p_j, p_i) - \left(e_j \overline{\eps} \sum_{l=1}^{j-1}\lambda(\tilde p_j, \tilde p_l) B^i_l \right) - e_j y_j \overline{\eps} B^i_j \right),\]
    where we have used the identity $\overline \eps \eps = 1$ several times.

    Next, we show that the projection of $\tilde{\phi}(v_1, \ldots, v_k)$ to~$H^g$ is $\lambda$-unimodular. For this we explain how the the block $A_{\tilde \phi(v_1, \ldots, v_k)}$ is obtained from the block $A_{(v_1, \ldots, v_k)}$ and show that the matrix part of the block $A_{\tilde \phi(v_1, \ldots, v_k)}$ is unimodular. Adding $\sum_{j=1}^{g} - \overline \eps \tilde p_j B^i_j$ to $v_i$ for each $i$ corresponds to changing only the last row of the block $A_{(v_1, \ldots, v_k)}$ and so doesn't affect its matrix part. Adding $\sum_{j=1}^{g} e_j \lambda(\tilde p_j, p_i)$ to $v_i$ for $1 \leq i \leq k$ corresponds to the following multiplication on the level of blocks:
        \[\left(\begin{tabular}{cc|c}
        \multicolumn{2}{c|}{\multirow{5}{*}{$\mathds{1}_{2g}$}} & $\tilde p_1$ \\
        \multicolumn{2}{c|}{} & $0$ \\
        \multicolumn{2}{c|}{} & $\vdots$ \\
        \multicolumn{2}{c|}{} & $\tilde p_g$ \\
        \multicolumn{2}{c|}{} & $0$ \\
        \hline
        \multicolumn{2}{c|}{$0$} & $1$
    \end{tabular}\right)
    \cdot A_{(v_1, \ldots, v_k)}.\]
    As we have seen above this is $\begin{pmatrix} B \\ u \end{pmatrix}$ with $B$ a~unimodular matrix. Adding the terms
        \[\sum_{j=1}^{g} - e_j \overline{\eps} \sum_{l=1}^{j-1}\lambda(\tilde p_j, \tilde p_l) B^i_l \text{ and } \sum_{j=1}^{g} - e_j y_j \overline{\eps} B^i_j\]
    corresponds to multiplying the block $A_{(v_1, \ldots, v_k)}$ from the left by matrices of the forms
    \[\left(\begin{tabular}{c|c}
            $C_1$ & $0$ \\
            \hline
            $0$ & $1$ \\
    \end{tabular}\right) \text{ and }
    \left(\begin{tabular}{c|c}
            $C_2$ & $0$ \\
            \hline
            $0$ & $1$ \\
    \end{tabular}\right)\]
    respectively, where $C_1$ is a~lower triangular matrix with $1$'s on the diagonal and $C_2$ is an~upper triangular matrix with $1$'s on the diagonal. In particular, both $C_1$ and $C_2$ are invertible. Note that all of the three above steps only change the coefficient of the~$e_i$, by adding on multiples of the coefficients of the $f_i$ and the last row. Therefore, applying $\tilde \phi$ to $(v_1, \ldots, v_k)$ corresponds to multiplying $A_{(v_1, \ldots, v_k)}$ from the left by the product of the above matrices:
    \[\left(\begin{tabular}{c|c}
            $C_2$ & $0$ \\
            \hline
            $0$ & $1$ \\
    \end{tabular}\right) \cdot
    \left(\begin{tabular}{c|c}
            $C_1$ & $0$ \\
            \hline
            $0$ & $1$ \\
    \end{tabular}\right) \cdot
    \left(\begin{tabular}{cc|c}
        \multicolumn{2}{c|}{\multirow{5}{*}{$\mathds{1}_{2g}$}} & $\tilde p_1$ \\
        \multicolumn{2}{c|}{} & $0$ \\
        \multicolumn{2}{c|}{} & $\vdots$ \\
        \multicolumn{2}{c|}{} & $\tilde p_g$ \\
        \multicolumn{2}{c|}{} & $0$ \\
        \hline
        \multicolumn{2}{c|}{$0$} & $1$
    \end{tabular}\right)
    \cdot A_{(v_1, \ldots, v_k)} =
    \begin{pmatrix} C_2 C_1 B \\ u \end{pmatrix}.\]
    Since $B$ is unimodular so is $C_2 C_1 B$. This corresponds to the projection of $\tilde{\phi}(v_1, \ldots, v_k)$ to~$H^g$ which is therefore also unimodular.

    Now applying \cite[Lemma~6.6]{MvdK} yields a~hyperbolic basis $\{\tilde e_1, \tilde f_1, \ldots, \tilde e_g, \tilde f_g\}$ of~$H^g$ such that
        \[\tilde \phi(v_1)|_{H^g}, \ldots, \tilde \phi(v_k)|_{H^g} \in \langle \tilde e_1, \tilde f_1, \ldots, \tilde e_k, \tilde f_k \rangle =: U.\]
    Note that this does not need to be the standard basis of $H^g$ hence we need to find an~automorphism~$\psi$ of~$H^g$ that sends the above basis $\tilde e_1, \tilde f_1, \ldots, \tilde e_g, \tilde f_g$ to the standard basis in~$H^g$. Then $\phi := (\mathds{1}_P \oplus \psi) \circ \tilde{\phi}$ will be the required automorphism. Let $V$ denote an~orthogonal complement of~$U$ in~$H^g$, i.e.\ $U \oplus V \cong H^g$. We have $g - k \geq \usr(R)$ and hence Proposition~\ref{PropCancellationH} implies $V \cong H^{g - k}$. Let $\psi$ denote the automorphism of~$H^g$ which sends $U$ to the first $k$ copies of~$H$ in~$H^g$ and $V$ to the last $g - k$ copies. Using the above definition of $\phi$ we then have $\phi(v_1, \ldots, v_k) \subseteq P \oplus H^k$ and the projection of~$\phi(v_1, \ldots, v_k)$ to~$H^k$ is unimodular.
\end{proof}

We get the following version of~\cite[Thm.~8.1]{MvdKV}, but phrased in terms of the unitary stable rank instead of the absolute stable rank.

\begin{cor} \label{CorThm8.1MvdKV}
    Let $r \in R$ and $(M, \lambda, \mu)$ be a~quadratic module satisfying $g(M) \geq \usr(R) + 1$. Then $U(M)$ acts transitively on the set of all $\lambda$-unimodular elements~$v$ in~$M$ satisfying $\mu(v) = r + \Lambda$.
\end{cor}

\begin{proof}
    For $g = g(M)$ there is a~quadratic module~$P$ such that $M \cong P \oplus H^g$. We write $e_1, f_1, \ldots, e_g, f_g$ for the basis of~$H^g$. We show that we can map a~$\lambda$-unimodular element~$v$ with $\mu(v) = r + \Lambda$ to $e_1 + r f_1$.

    By Lemma~\ref{Lem6.6MvdK} there is an~automorphism $\phi \in U(P \oplus H^g)$ such that $\phi(v) \subseteq P \oplus H$ and the projection of~$\phi(v)$ to the hyperbolic~$H$ is unimodular. Hence, by the transitivity condition~$(T_g)$ we can map the projection of~$\phi(v)$ (considered in $H^g$) to $e_1 + b^{\prime} f_1$ having the same length as the projection of~$\phi(v)$. Thus, we have mapped $v$ to the element $p + e_1 + b^{\prime} f_1$ for some element $p \in P$. Applying the orthogonal transvection $\tau(f_1, -p, x)$ for some $x \in \mu(p)$ maps $p + e_1 + b^{\prime} f_1$ to $e_1 + b f_1$, with $b = b^{\prime} + \overline{\eps} \lambda(p, p) - \overline{\eps}x$. We have
        \[r + \Lambda = \mu(v) = \mu(e_1 + b f_1) = b + \Lambda\]
    and
        \[\tau(f_1, 0, \eps(b - r))(e_1 + b f_1) = e_1 + r f_1. \qedhere \]
\end{proof}

By Lemma~\ref{LemNonSingular} this is a~generalisation of~\cite[Prop.~3.3]{GR-WStabNew} which treats the special case of quadratic modules over the integers. Note that our bound is slightly better than the bound given in the special case.

Adapting the proof of~\cite[Cor.~8.3]{MvdKV}, using Corollary~\ref{CorThm8.1MvdKV} instead of~\cite[Thm.~8.1]{MvdKV} yields the following improvement to Proposition~\ref{PropCancellationH}. Note that Step~$6$ of~\cite[Thm.~8.1]{MvdKV} still works in our setting.

\begin{cor} \label{CorCancellationHGeneral}
    Let $M$ and $N$ be quadratic modules and $M \oplus H \cong N \oplus H$. If $g(M) \geq \usr(R)$ then $M \cong N$.
\end{cor}

In contrast to Proposition~\ref{PropCancellationH}, both $M$ and $N$ can be general quadratic modules and, in particular, both can be non-hyperbolic modules. As in the previous corollary, this bound is slightly better than the bound given in~\cite[Prop.~3.4]{GR-WStabNew} which only treats the case $R = \Z$.

\subsection{Proof of Theorem~\ref{Thm7.4MvdK}}

For the proof of Theorem~\ref{Thm7.4MvdK} we follow a~strategy similar to the proof of~\cite[Thm.~7.4]{MvdK}. As we have seen in Lemma~\ref{LemNonSingular}, in the hyperbolic case every unimodular sequence is already $\lambda$-unimodular. In the case of general quadratic modules, however, a~unimodular sequence of length~$1$, $(v_1)$, need not be $\lambda$-unimodular and more generally, $(v_1, \ldots, v_k, u_1, \ldots, u_l)$ is not necessarily $\lambda$-unimodular, even if the individual sequences $(v_1, \ldots, v_k)$ and $(u_1, \ldots, u_l)$ are $\lambda$-unimodular. The following lemma, however, shows that in certain circumstances this implication is still valid.

\begin{lem} \label{LemLambdaUnimodularProperties}
    Let $(v_1, \ldots, v_k) \in \calU(M, \lambda)$ be a~$\lambda$-unimodular sequence in~$M$ and let $w_1, \ldots, w_k \in M$ be such that $\lambda(w_i, v_j) = \delta_{i, j}$.
    \begin{enumerate}
        \item We have $M = \langle v_1, \ldots, v_k \rangle \oplus \langle w_1, \ldots, w_k \rangle^{\perp}$ as a~direct sum of $R$-modules (i.e.\ the summands are not necessarily orthogonal with respect to $\lambda$).
        \item If $(u_1, \ldots, u_l) \in \calU(M, \lambda)$ is a~$\lambda$-unimodular sequence with $\lambda(w_i, u_j) = 0$ for all $i, j$ then the sequence $(v_1, \ldots, v_k, u_1, \ldots, u_l)$ is $\lambda$-unimodular.
        \item Let $u_i = x_i + y_i$ for $x_i \in \langle v_1, \ldots, v_k \rangle$ and $y_i \in \langle w_1, \ldots, w_k \rangle^{\perp}$. Then $(v_1, \ldots, v_k, u_1, \ldots, u_l)$ is $\lambda$-unimodular if and only if $(v_1, \ldots, v_k, y_1, \ldots, y_l)$ is $\lambda$-unimodular.
    \end{enumerate}
\end{lem}

\begin{proof}
    For (1) consider the map
        \[\bigoplus_{i = 1}^k \lambda(w_i, -) \colon M \longrightarrow R^k\]
    which sends $v_i$ to the $i$-th basis vector in $R^k$. The $v_i$ define a~splitting, and hence
        \[ M = \langle v_1, \ldots, v_k \rangle \oplus \Ker\left(\bigoplus_{i = 1}^k \lambda(w_i, -)\right) = \langle v_1, \ldots, v_k \rangle \oplus \langle w_1, \ldots, w_k \rangle^{\perp}.\]

    For (2) let $z_1, \ldots, z_l \in M$ such that $\lambda(z_i, u_j) = \delta_{i, j}$. Replacing $z_i$ by $z_i - \sum_{n=1}^{k}\lambda(z_i, v_n)w_n$ shows that the sequence $(v_1, \ldots, v_k, u_1, \ldots, u_l)$ is $\lambda$-unimodular since we have
    \begin{equation*}
        \begin{aligned}
            \lambda(w_i, v_j) & = \delta_{i, j} \qquad & \lambda \left(z_i - \sum_{n=1}^{k}\lambda(z_i, v_n)w_n, v_j \right) & = \lambda(z_i, v_j) - \lambda(z_i, v_j) = 0 \\
            \lambda(w_i, u_j) & = 0 \qquad & \lambda \left(z_i - \sum_{n=1}^{k}\lambda(z_i, v_n)w_n, u_j \right) & = \lambda(z_i, u_j) = \delta_{i, j}.
        \end{aligned}
    \end{equation*}

    To prove (3) we first assume that the sequence $(v_1, \ldots, v_k, y_1, \ldots, y_l)$ is $\lambda$-unimodular. Hence, there are $w_1, \ldots, w_k, z_1, \ldots, z_l \in M$ such that
    \begin{equation*}
        \begin{aligned}
            \lambda(w_i, v_j) & = \delta_{i, j} \qquad & \lambda (z_i, v_j) & = 0 \\
            \lambda(w_i, y_j) & = 0 \qquad & \lambda (z_i, y_j) & = \delta_{i, j}.
        \end{aligned}
    \end{equation*}
    Since $x_j \in \langle v_1, \ldots, v_k \rangle$ and $\lambda(z_i, v_j) = 0$ for all $i, j$ we have $\lambda(z_i, u_j) = \lambda(z_i, y_j) = \delta_{i, j}$. Replacing $w_i$ by $w_i - \sum_{n=1}^{l}\lambda(w_i, u_n)z_n$ shows that the sequence $(v_1, \ldots, v_k, u_1, \ldots, u_l)$ is $\lambda$-unimodular:
    \begin{equation*}
        \begin{aligned}
            \lambda\left(w_i - \sum_{n=1}^{l}\lambda(w_i, u_n)z_n, v_j\right) = \lambda(w_i, v_j) & = \delta_{i, j} \qquad & \lambda \left(z_i, v_j \right) & = 0 \\
            \lambda\left(w_i - \sum_{n=1}^{l}\lambda(w_i, u_n)z_n, u_j\right) = \lambda(w_i, u_j) - \lambda(w_i, u_j) & = 0 \qquad & \lambda \left(z_i, u_j \right) & = \delta_{i, j}.
        \end{aligned}
    \end{equation*}

    Now, assuming that $(v_1, \ldots, v_k, u_1, \ldots, u_l)$ is $\lambda$-unimodular we have $w_1, \ldots, w_k, z_1, \ldots, z_l \in M$ satisfying
    \begin{equation*}
        \begin{aligned}
            \lambda(w_i, v_j) & = \delta_{i, j} \qquad & \lambda (z_i, v_j) & = 0 \\
            \lambda(w_i, u_j) & = 0 \qquad & \lambda (z_i, u_j) & = \delta_{i, j}.
        \end{aligned}
    \end{equation*}
    As above we have $\lambda(z_i, y_j) = \lambda(z_i, u_j) = \delta_{i, j}$. Replacing $w_i$ by $w_i - \sum_{n=1}^{l}\lambda(w_i, y_n)z_n$ yields
    \begin{equation*}
        \begin{aligned}
            \lambda\left(w_i - \sum_{n=1}^{l}\lambda(w_i, y_n)z_n, v_j\right) = \lambda(w_i, v_j) & = \delta_{i, j} \qquad & \lambda \left(z_i, v_j \right) & = 0 \\
            \lambda\left(w_i - \sum_{n=1}^{l}\lambda(w_i, y_n)z_n, y_j\right) = \lambda(w_i, y_j) - \lambda(w_i, y_j) & = 0 \qquad & \lambda \left(z_i, y_j \right) & = \delta_{i, j},
        \end{aligned}
    \end{equation*}
    which shows that the sequence $(v_1, \ldots, v_k, y_1, \ldots, y_l)$ is $\lambda$-unimodular.
\end{proof}

To prove Theorem~\ref{Thm7.4MvdK} we need an~analogue of Theorem~\ref{ThmGLConnectivity} for the complex of $\lambda$-unimodular sequences in a~quadratic module. For this we use the following notation. Let $S \subseteq M$ be a~subset of a~quadratic module~$M$. We write $\calI(S, \mu)$ for the set of all elements $v \in S$ satisfying $\mu(v) = 0$.

\begin{thm} \label{ThmLambdaUnimodularConnectivity}
    Let $M = P \oplus H^g$ and $N$ be quadratic modules with $M \oplus H \subseteq N$.
    \begin{enumerate}
        \item $\calO\left(\calI(P \oplus \langle e_1, \ldots, e_g \rangle, \mu)\right) \cap \calU(N, \lambda)$ is $(g - \usr(R) - 1)$-connected,
        \item $\calO\left(\calI(P \oplus \langle e_1, \ldots, e_g \rangle, \mu)\right) \cap \calU(N, \lambda)_{(v_1, \ldots, v_k)}$ is $(g - \usr(R) - k - 1)$-connected for every sequence $(v_1, \ldots, v_k) \in \calU(N, \lambda)$.
    \end{enumerate}
\end{thm}

This is the natural generalisation of Theorem~\ref{ThmGLConnectivity} to the case of quadratic modules. (Only considering $N$'s of the form $M \oplus H^{\infty}$ is not sufficient for our proof of Lemma~\ref{Lem6.9MvdK}, see Remark~\ref{RmkMInftyN}.) We can write $N$ as $Q \oplus H^{g(M)} \oplus H^n$ for some $n \geq 1$, where $H^{g(M)}$ is the hyperbolic part of $M$ and $Q$ is some quadratic module. With this notation we have $P \subseteq Q \oplus H^{n-1}$, where $H^{n-1}$ denotes the last $n-1$ copies of~$H$ in $H^n \subseteq N$. In particular, $P$ is not necessarily contained in~$Q$.

The proof is an~adaptation of the proof of Theorem~\ref{ThmGLConnectivity} for which we use the following results.

\begin{prop} \label{PropLemma6.6ForV}
    Let $N$ be a~quadratic module with $H^k \subseteq N$ for $k \geq \usr(R)$. For a~$\lambda$-unimodular element $v \in N$ there is an~automorphism $\phi \in U(N)$ such that the projection of $\phi(v)$ to $H^k \subseteq N$ is $\lambda$-unimodular and for every subset $S \subseteq (H^k)^{\perp}$ the automorphism~$\phi$ fixes $S \oplus \langle e_1, \ldots, e_k \rangle$ as a~set.
\end{prop}

Note that in comparison with Lemma~\ref{Lem6.6MvdK} the bound for~$k$ in the above proposition is slightly lower than in the lemma. Hence, we get a~weaker conclusion here, having to restrict to more copies of the hyperbolic module~$H$ to get $\lambda$-unimodularity, and we cannot guarantee that the image of~$v$ under~$\phi$ lands outside certain copies of~$H$.

Saying that the automorphism~$\phi$ fixes $S \oplus \langle e_1, \ldots, e_k \rangle$ as a~set for every~$S \subseteq (H^k)^{\perp}$ is the same as saying that it is the identity on the associated graded for the filtration $0 \leq \langle e_1, \ldots, e_k \rangle \leq (H^k)^{\perp} \oplus \langle e_1, \ldots, e_k \rangle$ but we prefer the above formulation as this is of the form we use later on.

\begin{proof}
    We adapt the ideas of the first part of the proof of Lemma~\ref{Lem6.6MvdK}. We can write $N = Q \oplus H^k$ for some quadratic module~$Q$. Then
        \[v = p + \sum_{i = 1}^k e_i A_i + f_i B_i \quad \text{ for } p \in Q \text{ and } A_i, B_i \in R.\]
    As $v$ is $\lambda$-unimodular, there is
        \[w = q + \sum_{i = 1}^k e_i a_i + f_i b_i \quad \text{ for } q \in Q \text{ and } a_i, b_i \in R\]
    satisfying
    \begin{eqnarray*}
        1
        & =
        & \lambda(w, v) = (q, a_1, b_1, \ldots, b_k)
        \left(\begin{tabular}{c|cc|c}
            $\lambda|_Q$ & $0$ & $0$ & $\cdots$ \\
            \hline
            $0$ & $0$ & $1$ & \multirow{2}{*}{$\cdots$} \\
            $0$ & $\eps$ & $0$ & \\
            \hline
            $\vdots$ & \multicolumn{2}{c|}{$\vdots$} & $\ddots$
        \end{tabular}\right)
    \begin{pmatrix}
        p \\
        A_1 \\
        B_1 \\
        \vdots \\
        B_k
    \end{pmatrix} \\
        & =
    &   \lambda(q, p) + \sum_{i = 1}^k a_i B_i + \eps b_i A_i.
    \end{eqnarray*}
    Hence, using the notation from the proof of Lemma~\ref{Lem6.6MvdK}, the $2k \times 1$-block for~$Q$ associated to~$v$
    \[A_v = \begin{pmatrix}
        A_1 \\
        B_1 \\
        \vdots \\
        B_k \\
        \lambda(-, p)
    \end{pmatrix}\]
    is unimodular. Since $k  \geq \usr(R)$ by Proposition~\ref{PropThm1SnkVas} there are $p_1, \ldots, p_k \in Q$ such that for $m = (p_1, 0, \ldots, p_k, 0)$ we get
    \[\left(\begin{tabular}{c|c}
        $\mathds{1}_{2k}$ & $m^{\top}$ \\
        \hline
        $0$ & $1$
    \end{tabular}\right)
    \begin{pmatrix}
        A_1 \\
        B_1 \\
        \vdots \\
        B_k \\
        \lambda(-, p)
    \end{pmatrix}
    =
    \begin{pmatrix} b \\ u \end{pmatrix},\]
    where the vector~$b \in H^k$ is unimodular. As in the proof of Lemma~\ref{Lem6.6MvdK} this application of Proposition~\ref{PropThm1SnkVas} involves reordering the basis of the matrix part. Using the correspondence between elements in quadratic modules and their associated blocks explained in the proof of Lemma~\ref{Lem6.6MvdK}, multiplication with the above matrix is the required automorphism~$\phi$. Note that the bottom entry of $v$ stays fixed under $\phi$ and thus, for any $S \subseteq (H^k)^{\perp} = Q$ the automorphism~$\phi$ fixes $S \oplus \langle e_1, \ldots, e_g \rangle$ as a~set.
\end{proof}

\begin{lem} \label{LemLambdaUnimodularWitness}
    Let $N$ be a~quadratic module with $H^k \subseteq N$ for some $k$ and $v \in N$ so that the projection of~$v$ to~$H^{k - 1} \oplus 0 \subseteq H^k \subseteq N$ is $\lambda$-unimodular. There are $w \in \langle e_k, f_k \rangle$, $u \in H^{k - 1}$, and $x \in \mu(u)$ such that $\lambda(w, \tau(e_k, u, x)(v)) = 1$ and for every subset $S \subseteq (H^k)^{\perp}$ the transvection $\tau(e_k, u, x)$ fixes $S \oplus \langle e_1, \ldots, e_k \rangle$ as a~set.
\end{lem}

\begin{proof}
    Since the projection of~$v$ to~$H^{k - 1}$ is $\lambda$-unimodular there is an~element $z \in H^{k - 1}$ such that $\lambda(z, v) = 1$. For $u := (\overline{\lambda(f_k, v) - 1})\overline{\eps} z$ and any $x \in \mu(u)$ we have
    \begin{align*}
        \tau(e_k, u, x)(v) & = v + u \lambda(e_k, v) - e_k \overline{\eps} \lambda(u,v) - e_k \overline{\eps}x \lambda(e_k, v) \\
         & = v + u \lambda(e_k, v) + e_k (1 - \lambda(f_k, v) - \overline{\eps}x \lambda(e_k, v)).
    \end{align*}
    Since u is contained in $H^{k - 1}$ the second summand does not affect the coefficients of $e_k$ and $f_k$. The third summand changes the coefficient of $e_k$ to be $1 - \overline{\eps} x \lambda(e_k, v)$ and leaves all other coefficients fixed. Defining $w := \overline{x} \eps e_k + f_k$ we get
    \begin{align*}
        \lambda(w, \tau(e_k, u, x)(v)) & = \lambda(\overline{x}\eps e_k + f_k, (1 - \overline{\eps} x \lambda(e_k, v)) e_k + \lambda(e_k, v) f_k) \\
         & = \lambda(\overline{x}\eps e_k, \lambda(e_k, v) f_k) + \lambda(f_k, (1 - \overline{\eps}x\lambda(e_k, v)) e_k) \\
         & = \overline{\eps} x \lambda(e_k, v) + 1 - \overline{\eps} x \lambda(e_k, v) \\
         & = 1.
    \end{align*}
    Thus, choosing $u$, $x$, and $w$ as above shows the claim since the constructed transvection fixes $S \oplus \langle e_1, \ldots, e_k \rangle$ as a~set for every subset $S \subseteq (H^k)^{\perp}$.
\end{proof}

\begin{proof} [Proof of Theorem~\ref{ThmLambdaUnimodularConnectivity}.]
    Analogous to the proof of Theorem~\ref{ThmGLConnectivity} we will also show the following statements:
    \begin{enumerate}[\hspace{0.65cm}(a)]
    \item $\calO\Big(\calI\big(P \oplus (\langle e_1, \ldots, e_g \rangle \cup \langle e_1, \ldots, e_g \rangle + e_{g+1}), \mu \big)\Big) \cap \calU(N, \lambda)$ is $(g - \usr(R))$-connected,
    \item $\calO\Big(\calI \big(P \oplus (\langle e_1, \ldots, e_g \rangle \cup \langle e_1, \ldots, e_g \rangle + e_{g+1}), \mu \big)\Big) \cap \calU(N, \lambda)_{(v_1, \ldots, v_k)}$ is $(g - \usr(R) - k)$-connected for all $(v_1, \ldots, v_k)$ in $\calU(N, \lambda)$.
    \end{enumerate}
    Here, we write $N = Q \oplus H^g \oplus H$ for some quadratic module~$Q$ and $(e_{g+1}, f_{g+1})$ for the basis of the last copy of the hyperbolic~$H$ in~$N$.

    The proof is by induction on~$g$. Note that statements~(1), (2), and (b) all hold for $g < \usr(R)$ so we can assume $g \geq \usr(R)$. Statement~(a) holds for $g < \usr(R) - 1$ so we can assume $g \geq \usr(R) - 1$ when proving this statement. The structure of the proof is the same as in the proof of Theorem~\ref{ThmGLConnectivity}: We start by proving (b) which enables us to deduce (2). We will then prove statements~(1) and (a) simultaneously by applying statement~(2).

    In the following we write $E_g = \langle e_1, \ldots, e_g \rangle$.

    \underline{Proof of (b).} For $Y := P \oplus \left(E_g \cup (E_g + e_{g+1})\right)$ we write $F := \calO(\calI(Y, \mu)) \cap \calU(N, \lambda)_{(v_1, \ldots, v_k)}$. Let $d := g - \usr(R) - k$, so we have to show that $F$ is $d$-connected.

    For $g = \usr(R)$ the only case to consider is $k = 1$, where we have to show that $F$ is non-empty. By Proposition~\ref{PropLemma6.6ForV} there is an~automorphism $\phi \in U(N)$ such that the projection of $\phi(v_1)$ to $H^g \subseteq N$ is $\lambda$-unimodular and $\phi$ fixes $Y$ as a~set. Then the sequence $(\phi(v_1)|_{H^g}, e_{g+1})$ is $\lambda$-unimodular in~$N$. In particular, there is an~element $w_1 \in H^g$ such that $\lambda(w_1, \phi(v_1)|_{H_g}) = 1$ and $\lambda(w_1, e_{g+1}) = 0$. Now Lemma~\ref{LemLambdaUnimodularProperties}~(2) applied to $u_1 = e_{g+1}$ shows that the sequence $(\phi(v_1), e_{g+1})$ is $\lambda$-unimodular. Hence, the sequence $(v_1, \phi^{-1}(e_{g+1}))$ is also $\lambda$-unimodular. By construction we have $\phi^{-1}(e_{g+1}) \in Y$ and thus, $F$ is non-empty as it contains the element~$\phi^{-1}(e_{g+1})$.

    Now consider the case $g > \usr(R)$. By Proposition~\ref{PropLemma6.6ForV} there is an~automorphism~$\phi$ of $N$ such that the projection of $\phi(v_1)$ to $H^{g-1}$ is $\lambda$-unimodular. Using Lemma~\ref{LemLambdaUnimodularWitness} we get $u \in H^{g-1}$, $x \in \mu(u)$, and $w_1 \in \langle e_g, f_g \rangle$ such that $\lambda(w_1, \tau(e_g, u, x)(\phi(v_1))) = 1$. By construction, both $\phi$ and $\tau(e_g, u, x)$ fix $P \oplus (E_g \cup (E_g + e_{g + 1}))$ as a~set. Hence, the automorphism $\psi := \tau(e_g, u, x) \circ \phi$ defines an~isomorphism
        \[F = \calO(\calI(Y, \mu)) \cap \calU(N, \lambda)_{(v_1, \ldots, v_k)} \overset{\cong}{\longrightarrow} \calO(\calI(Y, \mu)) \cap \calU(N, \lambda)_{(\psi(v_1), \ldots, \psi(v_k))} = \psi(F).\]
    Writing $u_i := \psi(v_i)$ we have $\lambda(w_1, u_1) = 1$ with $w_1 \in \langle e_g, f_g \rangle$. This argument only works if $g > \usr(R)$ so we had to treat the case $g = \usr(R)$ separately.

    We want to use Lemma~\ref{Lem2.13vdK}~(1) to show that $\psi(F)$, and hence $F$, is $d$-connected. We define
        \[X := \calI(\{ v \in Y \ | \ v|_{\langle e_g, f_g \rangle} = 0 \}, \mu) = \calI(P \oplus (E_{g-1} \cup (E_{g-1} + e_{g+1})), \mu)\]
    and $u_i^{\prime} := u_i - u_1 \lambda(w_1, u_i)$ for $i > 1$, forcing $\lambda(w_1, u_i^{\prime}) = 0$. We have
    \begin{eqnarray*}
    \calO(X) \cap \psi(F) & = & \calO(X) \cap \calU(N, \lambda)_{(u_1, \ldots, u_k)} \\
    & = & \calO(\calI(P \oplus (E_{g-1} \cup (E_{g-1} + e_{g+1})), \mu)) \cap \calU(N, \lambda)_{(u_1, u_2^{\prime}, \ldots, u_k^{\prime})} \\
    & = & \calO(\calI(P \oplus (E_{g-1} \cup (E_{g-1} + e_{g+1})), \mu)) \cap \calU(N, \lambda)_{(u_2^{\prime}, \ldots, u_k^{\prime})},
    \end{eqnarray*}
    where the second equality holds as the span of $u_1, u_2^{\prime}, \ldots, u_k^{\prime}$ is the same as the span of $u_1, u_2, \ldots, u_k$ and the third equality can be seen as follows: The inclusion $\subseteq$ of the second line into the third is obvious. For the other inclusion, $\supseteq$, let $(x_1, \ldots, x_l) \in \calO(\calI(P \oplus (E_{g-1} \cup (E_{g-1} + e_{g+1})), \mu)) \cap \calU(N, \lambda)_{(u_2^{\prime}, \ldots, u_k^{\prime})}$. We have $\lambda(w_1, x_i) = 0$ since $w_1 \in \langle e_g, f_g \rangle$ and $\lambda(w_1, u_i^{\prime}) = 0$ by construction of the $u_i^{\prime}$. Thus, by Lemma~\ref{LemLambdaUnimodularProperties}~(2) the sequence $(x_1, \ldots, x_l, u_1, u_2^{\prime}, \ldots, u_k^{\prime})$ is $\lambda$-unimodular. In particular, the sequence $(x_1, \ldots, x_l)$ is an~element of $\calO(\calI(P \oplus (E_{g-1} \cup (E_{g-1} + e_{g+1})), \mu)) \cap \calU(N, \lambda)_{(u_1, u_2^{\prime}, \ldots, u_k^{\prime})}$.

    Thus, by the induction hypothesis $\calO(X) \cap \psi(F)$ is $d$-connected. Analogously, for $(w_1, \ldots, w_l) \in \psi(F) \setminus \calO(X)$ we get
    \begin{eqnarray*}
        \calO(X) \cap \psi(F)_{(w_1, \ldots, w_l)}
        & = & \calO(X) \cap \calU(N, \lambda)_{(u_1, \ldots, u_k, w_1, \ldots, w_l)} \\
        & = & \calO(\calI(P \oplus (E_{g-1} \cup (E_{g-1} + e_{g+1})), \mu)) \cap \calU(N, \lambda)_{(u_2^{\prime}, \ldots, u_k^{\prime}, w^{\prime}_1, \ldots, w^{\prime}_l)},
    \end{eqnarray*}
    which is $(d-l)$-connected by the induction hypothesis. Therefore, Lemma~\ref{Lem2.13vdK}~(1) shows that $\psi(F)$ is $d$-connected. Since $F$ and $\psi(F)$ are isomorphic, $F$ is therefore also $d$-connected.

    \underline{Proof of (2).} Let us write
        \[X := \calI\left(P \oplus (E_{g - 1} \cup (E_{g - 1} + e_g)), \mu \right).\]
    Then we have
    \begin{eqnarray*}
        & & \calO(X) \cap \left(\calO(\calI(P \oplus E_g, \mu)) \cap \calU(N, \lambda)_{(v_1, \ldots, v_k)}\right) \\
        & = & \calO\left( \calI(P \oplus (E_{g - 1} \cup (E_{g - 1} + e_g)), \mu ) \right) \cap \calU(N, \lambda)_{(v_1, \ldots, v_k)},
    \end{eqnarray*}
    which is $(d-1)$-connected by (b).

    Similarly, for $(w_1, \ldots, w_l) \in \calO(\calI(P \oplus E_g, \mu)) \cap \calU(N, \lambda)_{(v_1, \ldots, v_k)} \setminus \calO(X)$ we have
    \begin{eqnarray*}
        & & \calO(X) \cap \left(\calO(\calI(P \oplus E_g, \mu) \cap \calU(N, \lambda)_{(v_1, \ldots, v_k)}\right)_{(w_1, \ldots, w_l)} \\
        & = & \calO(X) \cap \calU(N, \lambda)_{(v_1, \ldots, v_k, w_1, \ldots, w_l)},
    \end{eqnarray*}
    which is $(d - l - 1)$-connected by the above. Hence, by Lemma~\ref{Lem2.13vdK}~(1) the claim follows.

    \underline{Proof of (1) and (a).} Note that we now only assume $g \geq \usr(R) - 1$. By induction let us assume that statement~(a) holds for $P \oplus (E_{g - 1} \cup (E_{g-1} + e_g))$ and we want to show it for $P \oplus (E_g \cup (E_g + e_{g+1}))$. Before we finish the induction for~(a) we will show that this already implies statement~(1) for $P \oplus E_g$. For this let $X$ be as in the proof of~(2) and $d := g - \usr(R)$. Then
    \begin{eqnarray*}
        & & \calO(X) \cap \left(\calO(\calI(P \oplus E_g, \mu)) \cap \calU(N, \lambda)\right) \\
        & = & \calO(\calI(P \oplus (E_{g - 1} \cup (E_{g - 1} + e_g)), \mu)) \cap \calU(N, \lambda)
    \end{eqnarray*}
    is $(d - 1)$-connected by (a). The complex $\calO(X) \cap \left(\calO(\calI(P \oplus E_g, \mu)) \cap \calU(N, \lambda)\right)_{(v_1, \ldots, v_m)}$ is $(d - m - 1)$-connected as we have already shown in the proof of~(2). Thus, $\calO(\calI(P \oplus E_g, \mu)) \cap \calU(N, \lambda)$ is $(g - \usr(R) - 1)$-connected by Lemma~\ref{Lem2.13vdK}~(1) which proves statement~(1).

    To prove (a) we will apply Lemma~\ref{Lem2.13vdK}~(2) for $X = \calI(P \oplus E_g, \mu)$ and $y_0 = e_{g+1}$. Consider
        \[(v_1, \ldots, v_k) \in \calO(\calI(P \oplus (E_g \cup (E_g + e_{g+1})), \mu)) \cap \calU(N, \lambda) \setminus \calO(X).\]
    Without loss of generality we may suppose that $v_1 \notin X$ as otherwise we can permute the $v_i$. By definition of~$X$ the coefficient of the $e_{g+1}$-coordinate of~$v_1$ is therefore $1$. Using Lemma~\ref{LemLambdaUnimodularProperties}~(2) as in part~(b) above we have
        \[\calO(X) \cap \calO(\calI(P \oplus (E_g \cup (E_g + e_{g+1})), \mu)) \cap \calU(N, \lambda)_{(v_1, \ldots, v_k)} = \calO(X) \cap \calU(N, \lambda)_{(v^{\prime}_2, \ldots, v^{\prime}_k)},\]
    where $v^{\prime}_i := v_i - v_1 \lambda(f_{g+1}, v_i)$ is chosen so that the coefficient of the $e_{g+1}$-coordinate of~$v^{\prime}_i$ is~$0$ for all~$i > 1$. This is $(d - k)$-connected by~(1) for $k = 1$ and by~(2) for $k \geq 2$. By construction we have
        \[\calO(X) \cap \calO(\calI(P \oplus (E_g \cup (E_g + e_{g+1})), \mu)) \cap \calU(N, \lambda) \subseteq (\calO(\calI(P \oplus (E_g \cup (E_g + e_{g+1})), \mu)) \cap \calU(N, \lambda))_{(e_{g+1})}\]
    and thus, Lemma~\ref{Lem2.13vdK}~(2) implies that $\calO(\calI(P \oplus (E_g \cup (E_g + e_{g+1})), \mu)) \cap \calU(N, \lambda)$ is $(g - \usr(R))$-connected which proves~(a).

    Note that when showing statement~(a) for $P \oplus (E_g \cup (E_g + e_{g+1}))$ we only used statement~(1) for $P \oplus E_g$ which follows from (a) for $P \oplus (E_{g-1} \cup (E_{g-1} + e_g))$ so this is indeed a~valid induction to show both statements~(1) and (a).
\end{proof}

In the following we write $\calU(M, \lambda, \mu) := \calO(\calI(M, \mu)) \cap \calU(M, \lambda)$.

\begin{cor} \label{CorLem6.8MvdK}
    Let $M$ and $N$ be quadratic modules with $M \oplus H \subseteq N$.
    \begin{enumerate}
        \item $\calO(M) \cap \calU(N, \lambda, \mu)$ is $(g(M) - \usr(R) - 1)$-connected,
        \item $\calO(M) \cap \calU(N, \lambda, \mu)_v$ is $(g(M) - \usr(R) - |v| - 1)$-connected for every $v \in \calU(N, \lambda, \mu)$,
        \item $\calO(M) \cap \calU(N, \lambda, \mu) \cap \calU(N, \lambda)_v$ is $(g(M) - \usr(R) - |v| - 1)$-connected for every $v \in \calU(N, \lambda)$.
    \end{enumerate}
\end{cor}

For the special case where the quadratic module $M$ is a~direct sum of hyperbolic modules~$H^n$, Corollary~\ref{CorLem6.8MvdK} has been proven by Mirzaii--van der Kallen in~\cite[Lemma~6.8]{MvdK}

\begin{proof}
    We write $g = g(M)$ and $M = P \oplus H^g$.

    For (1) let $W := \calI(P \oplus \langle e_1, \ldots, e_g \rangle, \mu)$ and $F:= \calO(M) \cap \calU(N, \lambda, \mu)$. Then we have
        \[\calO(W) \cap F = \calO(W) \cap \calU(N, \lambda) \text{ and } \calO(W) \cap F_u = \calO(W) \cap \calU(N, \lambda)_u,\]
    for every $u \in \calU(M, \lambda, \mu)$. Thus, by Theorem~\ref{ThmLambdaUnimodularConnectivity} the poset $\calO(W) \cap F$ is $(g - \usr(R) -1)$-connected and $\calO(W) \cap F_u$ is $(g - \usr(R) - |u| - 1)$-connected. Now, by Lemma~\ref{Lem2.13vdK}~(1) the poset~$F$ is $(g - \usr(R) - 1)$-connected.

    To show (3) we choose $W$ as above and $F$ as the complex $\calO(M) \cap \calU(N, \lambda, \mu) \cap \calU(N, \lambda)_v$. As before, using Lemma~\ref{Lem2.13vdK}~(1) yields the claim. Note that statement~(2) is a~special case of statement~(3).
\end{proof}

\begin{lem} \label{Lem6.9MvdK}
    For an~element $(v_1, \ldots, v_k) \in \calU(M, \lambda, \mu)$ the poset $\calO(\langle v_1, \ldots, v_k \rangle^{\perp}) \cap \calU(M, \lambda, \mu)_{(v_1, \ldots, v_k)}$ is $(g(M) - \usr(R) - k - 1)$-connected, where $\perp$ denotes the orthogonal complement with respect to $\lambda$.
\end{lem}

For the special case where $M$ is a~sum of hyperbolic modules~$H^n$ this has been done by Mirzaii-van der Kallen in~\cite[Lemma~6.9]{MvdK}.

\begin{proof}
    Let $g = g(M)$ and $M = P \oplus H^g$. By Lemma~\ref{Lem6.6MvdK} we can assume without loss of generality that $v_1, \ldots, v_k \in P \oplus H^k$ and the projection to the hyperbolic $H^k$ is $\lambda$-unimodular. In particular, there are $w_1, \ldots, w_k \in H^k$ such that $\lambda(w_i, v_j) = \delta_{i, j}$.
    Defining
        \[W := \calI(\langle v_1, \ldots, v_k, w_1, \ldots, w_k \rangle^{\perp}, \mu) \text{ and } F := \calO(\langle v_1, \ldots, v_k \rangle^{\perp}) \cap \calU(M, \lambda, \mu)_{(v_1, \ldots, v_k)}\]
    we have
        \[\calO(W) \cap F = \calO(W) \cap \calU(M, \lambda, \mu)_{(v_1, \ldots, v_k)} = \calO(W) \cap \calU(M, \lambda, \mu),\]
    where the second equality holds by Lemma~\ref{LemLambdaUnimodularProperties}~(2) as $W \subseteq \langle w_1, \ldots, w_k \rangle^{\perp}$. By construction we have $H^{g-k} \subseteq W$. Hence, $\calO(W) \cap F$ is $(g - k - \usr(R) - 1)$-connected by Lemma~\ref{CorLem6.8MvdK}~(1). By Lemma~\ref{LemLambdaUnimodularProperties}~(1) we can write $M = \langle v_1, \ldots, v_k \rangle \oplus \langle w_1, \ldots, w_k \rangle^{\perp}$, where we mean a~direct sum of $R$-modules. Consider $(u_1, \ldots, u_l) \in F \setminus \calO(W)$. We can write $u_i = x_i + y_i$ for $x_i \in \langle v_1, \ldots, v_k \rangle$ and $y_i \in \langle w_1, \ldots, w_k \rangle^{\perp}$. Note that $(y_1, \ldots, y_l)$ is in $\calU(M, \lambda)$ but not necessarily in $\calU(M, \lambda, \mu)$. Using Lemma~\ref{LemLambdaUnimodularProperties}~(2) and (3) we have
        \[\calO(W) \cap F_{(u_1, \ldots, u_l)} = \calO(W) \cap \calU(M, \lambda, \mu) \cap \calU(M, \lambda)_{(y_1, \ldots, y_l)}\]
    which is $(g - k - \usr(R) - l - 1)$-connected by Lemma~\ref{CorLem6.8MvdK}~(3). Using Lemma~\ref{Lem2.13vdK}~(1) now finishes the proof.
\end{proof}

\begin{rmk} \label{RmkMInftyN}
    \leavevmode
    \begin{enumerate}
        \item We could apply Corollary~\ref{CorLem6.8MvdK}~(2) directly to $\calO(\langle v_1, \ldots, v_k \rangle^{\perp}) \cap \calU(M, \lambda, \mu)_{(v_1, \ldots, v_k)}$ using that $H^{g(M) - k} \subseteq \langle v_1, \ldots, v_k \rangle^{\perp}$. However, this would only imply that the complex is $(g(M) - \usr(R) - 2k - 1)$-connected.
        \item In the proof of Lemma~\ref{Lem6.9MvdK} we cannot assume that the $y_i$'s lie in $\langle v_1, \ldots, v_k, w_1, \ldots, w_k \rangle^{\perp} \oplus H^{\infty}$. Hence, we need Theorem~\ref{ThmLambdaUnimodularConnectivity} in the generality it is stated.
    \end{enumerate}
\end{rmk}

Let $V$ be a~set and $F \subseteq \calO(V)$. For a~non-empty set~$S$ we define the poset $F \langle S \rangle$ as
    \[F \langle S \rangle := \{((v_1, s_1), \ldots, (v_k, s_k)) \in \calO(V \times S) \ | \ (v_1, \ldots, v_k) \in F\}.\]

\begin{lem} \label{Lem7.2MvdK}
    Let $g(M) \geq \usr(R) + k$. For $((v_1, w_1), \ldots, (v_k, w_k)) \in \calH \calU(M)$ we define $V := \langle v_1, \ldots, v_k \rangle$, $W := \langle w_1, \ldots, w_k \rangle$, and $Y := V^{\perp} \cap W^{\perp}$. Then
    \begin{enumerate}
        \item $\calI \calU(M)_{(v_1, \ldots, v_k)} \cong \calI \calU(Y)\langle V \rangle$,
        \item $\calH \calU(M) \cap \calM \calU(M)_{((v_1, 0),  \ldots, (v_k, 0))} \cong \calH \calU(X)\langle V \times V \rangle$,
        \item $\calH \calU(M)_{((v_1, w_1), \ldots, (v_k, w_k))} \cong \calH \calU(Y)$.
    \end{enumerate}
\end{lem}

For the case of hyperbolic modules this has been done in~\cite[Lemma~7.2]{MvdK}.

\begin{proof}
    We follow the proofs of~\cite[Lemma~3.4]{CharneyThmVogtmann} and \cite[Thm.~3.2]{CharneyThmVogtmann}.

    For (1) note that $\calI \calU(M)_{(v_1, \ldots, v_k)} \subseteq \calO(V^{\perp})$. Let $(u_1, \ldots, u_l) \in \calO(V^{\perp})$. We have $V^{\perp} = V \oplus Y$ by Lemma~\ref{LemLambdaUnimodularProperties}~(1) and therefore $u_i = x_i + y_i$ for some $x_i \in V$ and $y_i \in Y$. By Lemma~\ref{LemLambdaUnimodularProperties}~(3) the sequence $(u_1, \ldots, u_l, v_1, \ldots, v_k)$ is $\lambda$-unimodular if and only if the sequence $(y_1, \ldots, y_l, v_1, \ldots, v_k)$ is $\lambda$-unimodular, which holds if and only if the sequence $(y_1, \ldots, y_l)$ is $\lambda$-unimodular by Lemma~\ref{LemLambdaUnimodularProperties}~(2). Furthermore, we have $\mu(u_i) = \mu(y_i)$ and $\lambda(u_i, u_j) = \lambda(y_i, y_j)$ since $(v_1, \ldots, v_k) \in \calI \calU(M)$. Therefore, $\langle u_1, \ldots, u_l, v_1, \ldots, v_k \rangle$ is isotropic if and only if $\langle u_1, \ldots, u_l \rangle$ is isotropic and we get an~isomorphism
    \begin{eqnarray*}
        \calI \calU(M)_{(v_1, \ldots, v_k)} & \longrightarrow & \calI \calU(Y)\langle V \rangle \\
        (u_1, \ldots, u_l) & \longmapsto & ((y_1, x_1), \ldots, (y_l, x_l)).
    \end{eqnarray*}

    A~similar argument to the above for $\calH \calU(M) \cap \calM \calU(M)_{((v_1, 0),  \ldots, (v_k, 0))} \subseteq \calO(V^{\perp} \times V^{\perp})$ shows~(2). Statement~(3) holds by construction of~$Y$.
\end{proof}

The proof of~\cite[Thm.~7.4]{MvdK} uses the connectivity of the poset of isotropic $\lambda$-unimodular sequences in the hyperbolic module~$H^n$, $\calI \calU(H^n)$, given in~\cite[Thm.~7.3]{MvdK}. The following result is the analogous statement for general quadratic modules.

\begin{thm} \label{Thm7.3MvdK}
    The poset $\calI \calU(M)$ is $\bigl \lfloor \frac{g(M) - \usr(R) - 2}{2} \bigr \rfloor$-connected and for every $x \in \calI \calU(M)$ the poset $\calI \calU(M)_x$ is $\bigl \lfloor \frac{g(M) - \usr(R) - |x| - 2}{2} \bigr \rfloor$-connected.
\end{thm}

\begin{proof}[Outline of the proof.]
    The proof is analogous to the proof of~\cite[Thm.~7.3]{MvdK}, where we use Lemma~\ref{Lem6.6MvdK} instead of~\cite[Lemma~6.6]{MvdK}, Lemma~\ref{Lem6.9MvdK} instead of~\cite[Lemma~6.9]{MvdK}, and Lemma~\ref{Lem7.2MvdK} instead of~\cite[Lemma~7.2]{MvdK}. Note that \cite[Lemma~7.1]{MvdK} can easily be seen to hold in case of general quadratic modules.
\end{proof}

\begin{proof}[Outline of the proof of~Theorem~\ref{Thm7.4MvdK}]
    The proof is analogous to the proof of~\cite[Thm.~7.4]{MvdK}. The only changes that need to be made are the modifications described in the proof of~Theorem~\ref{Thm7.3MvdK} as well as using Theorem~\ref{Thm7.3MvdK} instead of~\cite[Thm.~7.3]{MvdK}.
\end{proof}

\subsection{Homological Stability} \label{SecHomStabU}

We now show homological stability for unitary groups over quadratic modules (Theorem~\ref{ThmMainU}). This induces in particular Theorem~B. As in the previous chapter we use the machinery of Randal-Williams--Wahl~\cite{R-WWahl}. Let $(R, \eps, \Lambda)$-Quad be the groupoid of quadratic modules over~$(R, \eps, \Lambda)$ and their isomorphisms. We write $f(R, \eps, \Lambda)$-Quad for the full subcategory on those quadratic modules which are finitely generated as $R$-modules. Since this is a~braided monoidal category it has an~associated pre-braided category $Uf(R, \eps, \Lambda) \text{-} \Quad$.

By Corollary~\ref{CorCancellationHGeneral} and \cite[Thm.~1.8~(a)]{R-WWahl} the category $Uf(R, \eps, \Lambda) \text{-} \Quad$ is locally homogeneous at~$(M, H)$ for $g(M) \geq \usr(R) + 1$. Axiom LH3 is verified by the following Lemma which for the special case of hyperbolic modules is shown in~\cite[Lemma~5.13]{R-WWahl}.

\begin{lem} \label{LemLH3U}
    Let $M$ be a~quadratic module with $g(M) \geq \usr(R) + 1$. Then the semisimplicial set $W_n(M, H)_{\bullet}$ is $\bigl \lfloor \frac{n + g(M) - \usr(R) - 3}{2}\bigr \rfloor$-connected.
\end{lem}

\begin{proof}
    As in the proof of~\cite[Lemma~5.13]{R-WWahl}, the poset of simplices of the semisimplicial set $W_n(M, H)_{\bullet}$ is equal to the poset $\calH \calU(M \oplus H^n)$ considered in Section~\ref{SecCpxU}. Hence, they have homeomorphic geometric realisations. The claim now follows from Theorem~\ref{Thm7.4MvdK}.
\end{proof}

Applying Theorems~\cite[Thm.~3.1]{R-WWahl}, \cite[Thm.~3.4]{R-WWahl} and \cite[Thm.~4.20]{R-WWahl} to the quadratic module $(Uf(R, \eps, \Lambda) \text{-} \Quad , \oplus, 0)$ yields the following theorem which directly implies Theorem~B.

\begin{thm} \label{ThmMainU}
    Let $M$ be a~quadratic module satisfying $g(M) \geq \usr(R) + 1$. For a~coefficient system $F \colon Uf(R, \eps, \Lambda) \text{-} \Quad \rightarrow \Z\text{-}\Mod$ of degree~$r$ at~$0$ in the sense of~\cite[Def,.~4.10]{R-WWahl}. Then for $s = g(M) - \usr(R)$ the map
        \[H_k(U(M); F(M)) \rightarrow H_k(U(M \oplus H); F(M \oplus H))\]
    is
    \begin{enumerate}
        \item an~epimorphism for $k \leq \frac{s - 1}{2}$ and an~isomorphism for $k \leq \frac{s - 2}{2}$, if $F$ is constant,
        \item an~epimorphism for $k \leq \frac{s - r - 1}{2}$ and an~isomorphism for $k \leq \frac{s - r - 3}{2}$, if $F$ is split polynomial in the sense of~\cite{R-WWahl},
        \item an~epimorphism for $k \leq \frac{s - 1}{2} - r$ and an~isomorphism for $k \leq \frac{s - 3}{2} - r$.
    \end{enumerate}
    For the commutator subgroup $U(M)^{\prime}$ we get that the map
        \[H_k(U(M)^{\prime}; F(M)) \rightarrow H_k(U(M \oplus H)^{\prime}; F(M \oplus H))\]
    is
    \begin{enumerate} \setcounter{enumi}{3}
        \item an~epimorphism for $k \leq \frac{s - 1}{3}$ and an~isomorphism for $k \leq \frac{s - 3}{3}$, if $F$ is constant,
        \item an~epimorphism for $k \leq \frac{s - 2r - 1}{3}$ and an~isomorphism for $k \leq \frac{s - 2r - 4}{3}$, if $F$ is split polynomial in the sense of~\cite{R-WWahl},
        \item an~epimorphism for $k \leq \frac{s - 1}{3} - r$ and an~isomorphism for $k \leq \frac{s - 4}{3} - r$.
    \end{enumerate}
\end{thm}

\section{Homological Stability for Moduli Spaces of High Dimensional Manifolds} \label{ChHomStabModuli}

Let $P$ be a~closed $(2n-1)$-dimensional manifold, and let $W$ and $M$ be compact connected $2n$-dimensional manifolds with identified boundaries $\partial W = P = \partial M$. In this chapter we follow Galatius--Randal-Williams~\cite{GR-WStabNew}. All statements and definitions are contained in the previous version, however, we use the numbering of what we understand will be the final version. We say that $M$ and $W$ are \textit{stably diffeomorphic relative to~$P$} if there is a~diffeomorphism
    \[W \# W_g \cong M \# W_h\]
relative to~$P$, for some $g, h \geq 0$, where $W_g := \#_g (\Sp^n \times \Sp^n)$ for $g \geq 0$. Let $\calM^{\st}(W)$ denote the set of $2n$-dimensional submanifolds $M \subset (-\infty, 0] \times \R^{\infty}$ such that
\begin{enumerate}
  \item $M \cap (\{0\} \times \R^{\infty}) = \{0\} \times P$ and $M$ contains $(-\eps, 0] \times P$ for some $\eps > 0$,
  \item the boundary of~$M$ is precisely $\{0\} \times P$, and
  \item $M$ is stably diffeomorphic to~$W$ relative to~$P$.
\end{enumerate}
We use the topology on~$\calM^{\st}(W)$ as described in~\cite[Ch.~6]{GR-WStabNew}. We write $\calM(W)$ for the model of the classifying space $B\Diff_{\partial}(W)$ defined in~\cite{GR-WStabNew}, which as a~set is the subset of $\calM^{\st}(W)$ given by those submanifolds that are diffeomorphic to~$W$ relative to~$P$. With this notion we have
    \[\calM^{\st}(W) = \bigsqcup_{[T]} \calM(T),\]
where the union is taken over the set of compact manifolds~$T$ with boundary~$\partial T = P$, which are stably diffeomorphic to~$W$ relative to~$P$, one in each diffeomorphism class relative to~$P$. The stabilisation map is the same as considered in~\cite{GR-WStabNew} and is given as follows: We choose a~submanifold $S \subset [-1, 0] \times \R^{\infty}$ with collared boundary $\partial S = \{-1, 0\} \times P = S \cap (\{-1, 0\}\times \R^{\infty})$, such that $S$ is diffeomorphic relative to its boundary to $([-1, 0] \times P) \# W_1$. If $P$ is not path connected, we also choose in which path component to perform the connected sum. Gluing $S$ then induces the self-map
\begin{eqnarray*} 
    s = - \cup S \colon \calM^{\st}(W)
    & \longrightarrow
    & \calM^{\st}(W) \\
    M
    & \longmapsto
    & (M - e_1) \cup S, \nonumber
\end{eqnarray*}
that is, translation by one unit in the first coordinate direction followed by union of submani\-folds of $(-\infty, 0] \times \R^{\infty}$.
Note that by construction we have $M \cup_P S \cong M \# W$ relative to~$P$, and hence $M \cup_P S$ is stably diffeomorphic to~$W$ if and only if $M$ is.

As in the previous chapters, we have a~notion of genus: Writing $W_{g, 1} := W_g \setminus \mathrm{int}(\D^{2n})$ the \textit{genus} of a~compact connected $2n$-dimensional manifold~$W$ is
    \[g(W) := \sup \{g \in \N \ | \ \text{there is an~embedding $W_{g, 1} \hookrightarrow W$}\}\]
and the \textit{stable genus} of~$W$ is
    \[\gbar(W) := \sup_{k \geq 0} \,  \{ g(W \# W_k) - k \ | \ k \in \N \}.\]
Note that since the map $k \mapsto g(W \# W_k) - k$ is non-decreasing and bounded above by $\frac{b_n(W)}{2}$, where $b_n(W)$ is the $n$-th Betti number of $W$, the above supremum is finite. The following theorem shows homological stability for the graded spaces (in the sense of~\cite[Def.~6.6]{GR-WStabNew}) $\calM^{\st}(W)_g \subset \calM^{\st}(W)$, which are those manifolds $M \in \calM^{\st}(W)$ satisfying $\gbar(M) = g$. Note that by definition of the stable genus, the map~$s$ defined above restricts to a~map $s \colon \calM^{\st}(W)_g \rightarrow \calM^{\st}(W)_{g+1}$. For the case of simply-connected compact manifolds Galatius--Randal-Williams have shown homological stability for the spaces $\calM(W)_g$ in~\cite[Thm.~6.3]{GR-WStabNew}.

\begin{thm} \label{Thm6.3GRW}
    Let $2n \geq 6$ and $W$ be a~compact connected manifold. Then the map
        \[s_* \colon H_k(\calM^{\st}(W)_g) \longrightarrow H_k(\calM^{\st}(W)_{g+1})\]
    is an~epimorphism for $k \leq \frac{g - \usr(\Z[\pi_1(W)])}{2}$ and an~isomorphism for $k \leq \frac{g - \usr(\Z[\pi_1(W)]) - 2}{2}$.
\end{thm}

This in particular implies that for any manifold~$W$ with boundary~$P$, the restriction
    \[s \colon \calM(W) \longrightarrow \calM(W \cup_P S)\]
induces an~epimorphism on homology in degrees satisfying $k \leq \frac{\gbar(W) -\usr(\Z[\pi_1(W)])}{2}$ and an~isomorphism in degrees satisfying $k \leq \frac{\gbar(W) -\usr(\Z[\pi_1(W)]) - 2}{2}$. Since $g(W) \leq \gbar(W)$ this implies Theorem~C.

Using Example~\ref{ExsUsr}~(\ref{ExsUsrVirtuallyPolycyclic}) we get a~special case of~Theorem~\ref{Thm6.3GRW}.

\begin{cor} \label{CorMainUVirtuallyPolycyclic}
    Let $2n \geq 6$ and $W$ be a~compact connected manifold whose fundamental group is virtually polycyclic of Hirsch length~$h$. Then the map
        \[s_* \colon H_k(\calM^{\st}(W)_g) \longrightarrow H_k(\calM^{\st}(W)_{g+1})\]
    is an~epimorphism for $k \leq \frac{g - h - 3}{2}$ and an~isomorphism for $k \leq \frac{g - h - 5}{2}$.
\end{cor}

This theorem applies in particular to all compact connected manifolds with finite fundamental group and more generally with finitely generated abelian fundamental group.

Another consequence of the above theorem is the following cancellation result which in the case of simply-connected manifolds has been done in~\cite[Cor.~6.3]{GR-WStabNew}. The statement is closely related to~\cite[Thm.~1.1]{CrowleySixt}.

\begin{cor} \label{CorCancellationModuliSpace}
    Let $2n \geq 6$ and $P$ be a~$(2n-1)$-dimensional manifold. Let $W$ and $W^{\prime}$ be compact connected manifolds with boundary~$P$ such that $W \# W_g \cong W^{\prime} \# W_g$ relative to~$P$ for some $g \geq 0$. If $\gbar(W) \geq \usr(\Z[\pi_1(W)]) + 2$, then $W \cong W^{\prime}$ relative to~$P$.
\end{cor}

\begin{proof}
    Analogous to the proof of~\cite[Cor.~6.4]{GR-WStabNew}, where we apply Theorem~\ref{Thm6.3GRW} instead of~\cite[Thm.~6.3]{GR-WStabNew}.
\end{proof}

The proof of Theorem~\ref{Thm6.3GRW} is analogous to that of~\cite[Thm.~6.3]{GR-WStabNew} which treats the case of simply-connected manifolds. The idea is to consider the group of immersions of $(\Sp^n \times \Sp^n) \setminus \mathrm{int}(\D^{2n})$ into a~manifold. Equipping this with a~bilinear form that counts intersections and a~function that counts self-intersections we get a~quadratic module. The precise construction is the content of the following section. The high connectivity shown in the previous chapter then implies a~connectivity statement for a~complex of geometric data associated to the manifold. This is the crucial result in order to show homological stability which we do in Section~\ref{SecProofThm6.3GRW}.

\subsection{Associating a~Quadratic Module to a~Manifold} \label{SecManifoldIntoQuadMod}

In order to relate the objects in this chapter to the algebraic objects considered in Section~\ref{SecCpxU} we want to associate to each compact connected $2n$-dimensional manifold~$W$ a~quadratic module~$(\calI^{\tn}_n(W), \lambda, \mu)$  with form parameter $((-1)^n, \Lambda_{\min})$. This will be a~$\Z[\pi_1(W, *)]$-module given by a~version of the group of immersed $n$-spheres in~$W$ with trivial normal bundle, with pairing given by the intersection form, and quadratic form given by counting self-intersections, both considered over the group ring~$\Z[\pi_1(W, *)]$. For the rest of this chapter we drop the basepoint~$*$ from the notation and just write $\pi_1(W)$.

To make this construction precise we fix a~framing $b_{\Sp^n \times \D^n} \in \mathrm{Fr}(\Sp^n \times \D^n)$  at the basepoint in $\Sp^n \times \D^n$ as defined in~\cite[Ch.~5]{GR-WStabNew}. We can now generalise \cite[Def.~5.2]{GR-WStabNew}, following the construction in the proof of~\cite[Thm.~5.2]{Wall}.

\begin{defn} \label{DefCalIfrLambdaMu}
    Let $2n \geq 6$ and $W$ be a~compact connected $2n$-dimensional manifold, equipped with a~\textit{framed basepoint}, i.e.\ a~point $b_W \in \mathrm{Fr}(W)$, and an~orientation compatible with~$b_W$.
    \begin{enumerate}
        \item We consider the ring $\Z[\pi_1(W)]$ with involution given by $\gbar := w_1(g) g^{-1} \in \Z[\pi_1(W)]$, where $w_1(g)$ is the first Stiefel--Whitney class of~$g$. Recall that the first Stiefel--Whitney class can be viewed as the homomorphism $\pi_1(W) \rightarrow \Z^{\times} = \{-1, 1\}$ which sends a~loop to $1$ if and only if it is orientation preserving.

            Let $\calI^{\fr}_n(W)$ be the set of regular homotopy classes of immersions \mbox{$i \colon \Sp^n \times \D^n \looparrowright W$} equipped with a~path in $\mathrm{Fr}(W)$ from $Di(b_{\Sp^n \times \D^n})$ to $b_W$. We write $\calI_n(W)$ for the set of regular homotopy classes of immersions $\Sp^n \looparrowright W$ equipped with a~path in~$W$ from a~fixed basepoint in~$\Sp^n$ to the basepoint~$*$ in~$W$. We define $\calI^{\tn}_n(W)$ to be the image of the map $\calI^{\fr}_n(W) \rightarrow \calI_n(W)$ which is given by forgetting the framing. Since an~immersion $\Sp^n \looparrowright W$ is framable if and only if it has a~trivial normal bundle, the set $\calI^{\tn}_n(W)$ is given by regular homotopy classes of immersions with a~trivial normal bundle.

            Using Smale-Hirsch immersion theory we can identify $\calI_n(W)$ with the $n$-th homotopy group of $n$-frames in~$W$. This induces an~(abelian) group structure on~$\calI_n(W)$. The $\pi_1(W)$-action is given by concatenating a~loop in~$W$ with the path corresponding to an~element in~$\calI_n(W)$ as described in~\cite[Thm.~5.2]{Wall}. Now, $\calI^{\tn}_n(W)$ is a~$\Z[\pi_1(W)]$-submodule of $\calI_n(W)$.
        \item Let $a, b \in \calI^{\tn}_n(W)$ be two immersed spheres, which we may suppose meet in general position, i.e.\ transversely in a~finite set of points. For a~point~$p$ in~$a$ let $\gamma_a(p)$ denote a~path from the basepoint~$*$ to $p$ in $a$. Since $2n \geq 6$ such a~path is canonical up to homotopy. For $p \in a \cap b$ we define $\gamma_{(a, b)}(p)$ to be the concatenation of $\gamma_a(p)$ followed by the inverse of ${\gamma_b(p)}$.

            Let us fix an~orientation of~$W$ at the basepoint~$*$ and transport the orientation to~$p$ along~$a$. Then $\eps_{(a, b)}(p)$ is defined to be the sign of the intersection of~$a$ and $b$ with respect to this orientation at~$p$. Given these notions we define a~map
            \begin{eqnarray*}
                \lambda \colon \calI^{\tn}_n(W) \times \calI^{\tn}_n(W)
                & \longrightarrow
                & \Z[\pi_1(W)] \\
                (a, b)
                & \longmapsto
                & \sum_{p \in a \cap b} \eps_{(a, b)}(p) \gamma_{(a, b)}(p).
            \end{eqnarray*}
        \item Let $a \in \calI^{\tn}_n(W)$ be an~immersed sphere in general position and let $p \in \Sp^n \times \{0\}$ be a~point in~$a$. We write $\gamma(p)$ for the path from the basepoint~$*$ to $p$ in the universal cover of the image of~$a$ in~$W$.

            At a~self-intersection point of~$a$ two branches of $a$ cross. By choosing an~order of these branches we can define $\eps(p, q)$ as above. Recall that $\Lambda_{\min} = \{\gamma - \eps \overline{\gamma} \ | \ \gamma \in \Z[\pi_1(W)]\}$. We define a~map
            \begin{eqnarray*}
                \mu \colon \calI^{fr}_n(W)
                & \longrightarrow
                & \Z[\pi_1(W)] / \Lambda_{\min} \\
                a
                & \longmapsto
                & \sum_{\substack{\{p, q\} \subset \Sp^n \\ i_a(p) = i_a(q) \\ p \neq q}} \eps(p, q) \gamma(p, q),
            \end{eqnarray*}
            where $i_a$ is an~immersion of $\Sp^n$ corresponding to~$a$ and $\gamma(p, q)$ is the loop in~$a$ based at the basepoint~$*$ given by the concatenation of $a(\gamma(p))$ and the inverse of $a(\gamma(q))$, see Figure~\ref{FigDefMu}. The definition of $\Lambda_{\min}$ guarantees that the order of the points $p, q$ is not relevant, i.e.\ we have $\eps(p, q) \gamma(p, q) \equiv \eps(q, p) \gamma (q, p) \mod \Lambda_{\min}$.
            \begin{figure}
                \centering
                \begin{minipage}{.5\textwidth}
                    \centering
                    \includegraphics[viewport=0 -1 192 151]{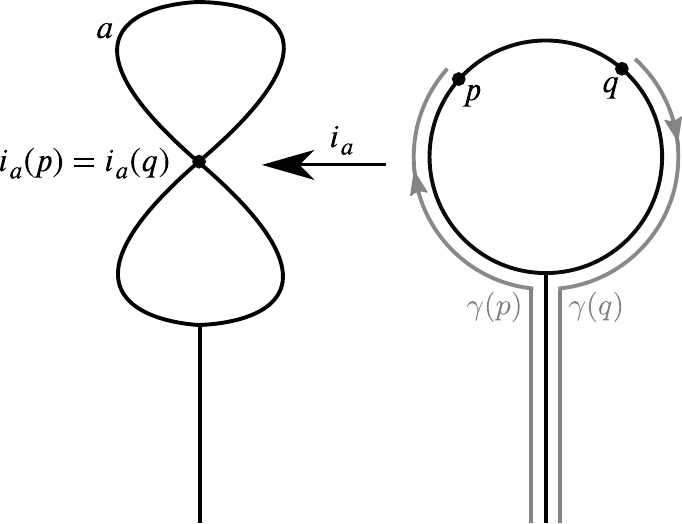}
                    \captionof{figure}{Definition of $\gamma(p, q)$.}
                    \label{FigDefMu}
                \end{minipage}%
                \begin{minipage}{.5\textwidth}
                    \centering
                    \includegraphics[viewport=0 -1 59 153]{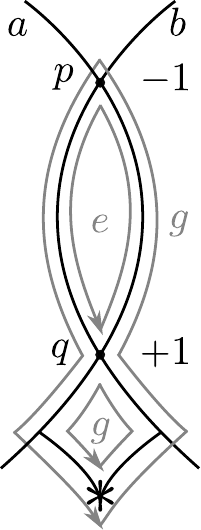}
                    \captionof{figure}{Using the Whitney trick.}
                    \label{FigWhitneyTrickLambda}
                \end{minipage}
            \end{figure}
    \end{enumerate}
\end{defn}

\begin{rmks}\label{RmkLambdaMu}
    \leavevmode
    \begin{enumerate}
      \item The (abelian) group structure on~$\calI^{\tn}_n(W)$ is given by forming the connected sum along the path as described in~\cite[Ch.~5]{Wall}.
      \item The proof of~\cite[Thm.~5.2~(i)]{Wall} shows that both maps~$\lambda$ and $\mu$ are well-defined.
      \item \label{RmkLambdaMuLambda} We show that we can always change $a$ by an~isotopy so that every point in~$a \cap b$ yields a~summand in $\lambda(a, b)$, i.e.\ so that no two intersection points give summands that cancel. The idea is to pair up intersection points that give the same element in~$\Z[\pi_1(W)]$ but with opposite signs, and to use the Whitney trick to kill these intersection points. Figure~\ref{FigWhitneyTrickLambda} shows a~sector of~$a$ and $b$ in~$W$ with two intersection points~$p$ and $q$.
          Both paths~$\gamma_{(a, b)}(p)$ and $\gamma_{(a, b)}(q)$ correspond to $g$ in $\pi_1(W)$ and the points~$p$ and $q$ have opposite signs. If this is the case the loop~$e$ is contractible. Hence, we can fill in a~$2$-disc and use the Whitney trick in order to move $a$ away from $b$ in the sector shown in the picture.

    \end{enumerate}
\end{rmks}

The subsequent lemma generalises \cite[Lemma~5.3]{GR-WStabNew}. The proof is analogous to the proof of~\cite[Lemma~5.3]{GR-WStabNew}, again using \cite[Thm.~5.2]{Wall}.

\begin{lem} \label{LemCalIfrQuadMod}
    The triple $(\calI^{\tn}_n(W), \lambda, \mu)$ is a~$((-1)^n, \Lambda_{\min})$-quadratic module.
\end{lem}

\subsection{Proof of Theorem~\ref{Thm6.3GRW}} \label{SecProofThm6.3GRW}

We denote by~$H$ the manifold we obtain from $W_{1, 1}$ by gluing $[-1, 0] \times \D^{2n-1}$ onto $\partial W_{1, 1}$ along an~oriented embedding
    \[\{-1\} \times \D^{2n-1} \longrightarrow \partial W_{1, 1}.\]
We choose this embedding once and for all. After smoothing corners, $H$ is diffeomorphic to~$W_{1, 1}$ but contains a~standard embedding of $[-1, 0] \times \D^{2n-1}$. By an~embedding of~$H$ into a~manifold~$W$ we mean an~embedding that maps $\{0\} \times \D^{2n-1}$ into $\partial W$ and the rest of~$H$ into the interior of~$W$.
We define the embeddings~$\overline e$ and $\overline f$ of~$\Sp^n$ into~$H$ as the inclusion $\Sp^n \hookrightarrow \Sp^n \times \D^n$ given by $x \mapsto (x, 0)$ followed by the maps~\cite[(5.2)]{GR-WStabNew} and \cite[(5.3)]{GR-WStabNew} respectively (see Figure~\ref{FigDefHef}).
\begin{figure}
    \includegraphics[viewport=0 -1 181 78]{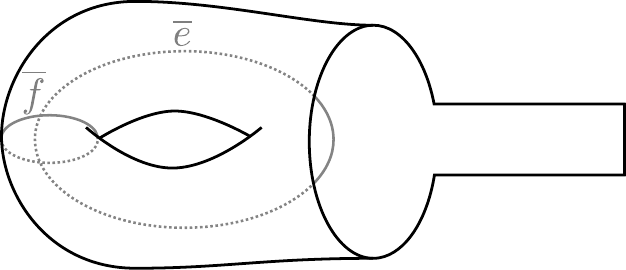}
    \captionof{figure}{Definition of $\overline e$ and $\overline f$ in $H$.}
    \label{FigDefHef}
\end{figure}
The embedding~$\overline e$ together with a~path in~$H$ from the basepoint of~$\overline e(\Sp^n)$ to the basepoint~$(0, 0)$ in~$[-1, 0] \times \D^{2n-1} \subseteq H$ defines an~element $e \in \calI^{\tn}_n(H)$. Since $H$ is simply-connected the choice of path is unique up to isotopy. Analogously, we get an~element~$f \in \calI^{\tn}_n(H)$. Hence, an~embedding~$\phi$ of~$H$ into~$W$ yields a~hyperbolic pair $\phi_*(e)$, $\phi_*(f)$ in $\calI^{\tn}_n(W)$. As described in~\cite[Ch.~5]{GR-WStabNew} this can be extended to a~map
\begin{equation*}
    K^{\delta}(W) \longrightarrow \calH \calU(\calI^{\tn}_n(W)),
\end{equation*}
where $K^{\delta}(W)$ denotes the simplicial complex as defined in~\cite[Def.~5.1]{GR-WStabNew} and $\calH \calU(\calI^{\tn}_n(W))$ is the simplicial complex defined in Section~\ref{SecCpxU}. We use this map to deduce the connectivity of $|K^{\delta}_{\bullet}(W)| = |K^{\delta}(W)|$ from the connectivity of $\calH \calU(\calI^{\tn}_n(W))$ which we have shown in Section~\ref{SecCpxU}. This is the content of the next theorem. For the case of simply-connected manifolds this has been done in~\cite[Lemma~5.5]{GR-WStabNew}, \cite[Thm.~5.6]{GR-WStabNew}, and \cite[Cor.~5.10]{GR-WStabNew}.

\begin{thm} \label{ThmAllKConnectivity}
    Let $2n \geq 6$ and $W$ be a~compact connected $2n$-dimensional manifold. Then the following spaces, defined in~\cite[Ch.~5]{GR-WStabNew}, are all $\bigl \lfloor \frac{\gbar(W) - \usr(\Z[\pi_1(W)]) - 3}{2} \bigr \rfloor$-connected:
    \begin{enumerate}
      \item $|K^{\delta}_{\bullet}(W)|$,
      \item $|K_{\bullet}(W)|$,
      \item $|\Kbar_{\bullet}(W)|$.
    \end{enumerate}
\end{thm}

For the proof of this theorem we want a~modified version of~Theorem~\ref{Thm7.4MvdK} using the following definition: The \textit{stable Witt index} of a~quadratic module~$M$ is
    \[\gbar(M) := \sup_{k \geq 0} \, \{g(M \oplus H^k) - k\}.\]
By definition we have $g(M) \leq \gbar(M)$ and if the stable Witt index is big enough we in fact have equality, as the following corollary shows.
\begin{lem} \label{LemGbarG}
    If $\gbar(M) \geq \usr(R)$ then we have $g(M) \geq \gbar(M)$.
\end{lem}

\begin{proof}
    For $g = \gbar(M)$ we know that $M \oplus H^k \cong P \oplus H^g \oplus H^k$ for some~$k$. If $k=0$ we immediately get $g(M) \geq g$. If $k > 0$ we get $M \oplus H^{k-1} \cong P \oplus H^g \oplus H^{k-1}$ by Corollary~\ref{CorCancellationHGeneral}. Applying this argument inductively then yields $g(M) \geq g$.
\end{proof}

Using the above correspondence between the Witt index and the stable Witt index we can now state Theorem~\ref{Thm7.4MvdK} in terms of the stable Witt index.

\begin{cor} \label{CorUConnectivityGbar}
    The poset $\calH \calU(M)$ is $\bigl \lfloor \frac{\gbar(M) - \usr(R) - 3}{2} \bigr \rfloor$-connected and $\calH \calU(M)_x$ is $\bigl \lfloor \frac{\gbar(M) - \usr(R) - |x| - 3}{2} \bigr \rfloor$-connected for every $x \in \calH \calU(M)$.
\end{cor}

\begin{rmk}
    Analogous to the above we can define the \textit{stable rank} of an~$R$-module~$M$ as
        \[\rkbar(M) := \sup_{k \geq 0} \, \{\rk(M \oplus R^k) - k \}.\]
    As in the case of the stable Witt index this coincides with the rank of~$M$ if $\rkbar(M) \geq \sr(R)$. This can be shown similarly to the proof of Lemma~\ref{LemGbarG} by inductively applying Theorem~\ref{ThmGLConnectivity} and Proposition~\ref{PropGLCancellation}. Using this we get a~version of Theorem~\ref{ThmGLConnectivity} in terms of the stable rank.
\end{rmk}

\begin{proof}[Proof of~Theorem~\ref{ThmAllKConnectivity}.]
    For $|K^{\delta}_{\bullet}(W)|$ the proof is analogous to the proof of~\cite[Lemma~5.5]{GR-WStabNew} and hence we just comment on the changes we have to make in order to show the above statement. Note that the complex $K^a(\calI^{\tn}_n(W), \lambda, \mu)$ as defined in~\cite[Def.~3.1]{GR-WStabNew} is the same as $\calH \calU(\calI^{\tn}_n(W))$.

    For $\gbar = \gbar (\calI^{fr}_n(W), \lambda, \mu)$ we have $\gbar(W) \leq \gbar$ and hence it is sufficient to show that $|K^{\delta}_{\bullet}(W)|$ is $\bigl \lfloor \frac{\gbar - \usr(\Z[\pi_1(W)]) - 3}{2}\bigr \rfloor$-connected.

    For $k \leq \frac{\gbar - \usr(\Z[\pi_1(W)]) - 3}{2}$ we consider a~map $f \colon \partial I^{k+1} \rightarrow |K^{\delta}_{\bullet}(W)|$, which, as in~\cite{GR-WStabNew}, we may assume is simplicial with respect to some piecewise linear triangulation $\partial I^{k+1} \cong |L|$. By Corollary~\ref{CorUConnectivityGbar} and composing with the map constructed above we get a~nullhomotopy $\overline{f} \colon I^{k+1} \rightarrow |\calH \calU(\calI^{\tn}_n(W))|$. We have to show that this lifts to a~nullhomotopy $F \colon I^{k+1} \rightarrow |K^{\delta}_{\bullet}(W)|$ of~$f$.

    By Corollary~\ref{CorUConnectivityGbar} the complex~$\calH \calU(\calI^{\tn}_n(W))$ is locally weakly Cohen-Macaulay (as defined in~\cite[Sec.~2.1]{GR-WStabNew}) of dimension $\Big \lfloor \frac{\gbar - \usr(\Z[\pi_1(W)])}{2} \Big \rfloor \geq k+1$. Hence, there is a~triangulation $I^{k+1} \cong |K|$ extending~$L$ which satisfies the same properties as in~\cite{GR-WStabNew}.

    We choose an~enumeration of the vertices in~$K$ as $v_1, \ldots, v_N$ such that the vertices in $L$ come before the vertices in $K \setminus L$. We inductively pick lifts of each $\overline f(v_i) \in \calH \calU(\calI^{\tn}_n(W))$ to a~vertex $F(v_i) \in K^{\delta}_{\bullet}(W)$ given by an~embedding $j_i \colon H \rightarrow W$ satisfying the properties~(i) and (ii) in the proof of~\cite[Lemma~5.5]{GR-WStabNew} which control how the images of every two such embeddings intersect. By construction, the vertices in~$L$ already satisfy the required properties~(i) and (ii), so we can assume that $\overline f(v_1), \ldots, \overline f(v_{i-1})$ have already been lifted to maps $j_1, \ldots, j_{i-1}$, satisfying properties~(i) and (ii). Then $v_i \in K \setminus L$ yields a~morphism of quadratic modules $\overline f(v_i) = h \colon H_{hyp} \rightarrow \calI^{\tn}_n(W)$, where $H_{hyp}$ is the hyperbolic module defined in Chapter~\ref{ChHomStabU}, which we want to lift to an~embedding~$j_i$ satisfying properties~(i) and (ii). The element~$h(e)$ is represented by an~immersion $x: \Sp^n \looparrowright W$ with trivial normal bundle satisfying $\mu(x) = 0$ and a~path in~$W$ from the basepoint of~$\Sp^n$ to the basepoint~$*$ of~$W$. By the Whitney trick (which works in our case, but we have to use it over the group ring $\Z[\pi_1(W)]$ as described in Remark~\ref{RmkLambdaMu}~(\ref{RmkLambdaMuLambda})) we can replace $x$ by an~embedding $j(e) \colon \Sp^n \hookrightarrow W$. Similarly, $h(f)$ yields an~embedding $j(f) \colon \Sp^n \hookrightarrow W$, along with another path in~$W$.

    Using the Whitney trick again, we can arrange for the embeddings $j(e)$ and $j(f)$ to intersect transversally in exactly one point. Hence, by picking a~trivialisation of their normal bundles, this induces an~embedding $W_{1, 1} \hookrightarrow W$. To extend this map to an~embedding $H \hookrightarrow W$ of manifolds, note that both $h(e)$ and $h(f)$ come with a~path to the basepoint. The proof in~\cite{GR-WStabNew} forgets both path and chooses a~new one later on (which works since $W$ is simply-connected and hence oriented). Instead, we can keep track of the path coming from $h(e)$. This can be viewed as an~embedding $[-1, 0] \times \{0\} \hookrightarrow W$. This then has a~thickening by definition which gives an~embedding $H \hookrightarrow W$. Analogous to the proof of~\cite[Lemma~5.5]{GR-WStabNew} we can show that the properties~(i) and (ii) hold, and hence conclude the connectivity range.

    The proof for the case $|K_{\bullet}(W)|$ is an~easy extension of the proof of~\cite[Thm.~5.6]{GR-WStabNew}, where we use Corollary~\ref{CorUConnectivityGbar} instead of~\cite[Thm.~3.2]{GR-WStabNew} and hence get a~slightly weaker connectivity range.

    The remaining case follows exactly as in~\cite[Cor.~5.10]{GR-WStabNew}.
\end{proof}

In the above proof, we have lifted the chosen nullhomotopy $\overline f \colon I^{k+1} \rightarrow |\calH \calU(\calI^{\tn}_n(W))|$ and do not have to use the ``spin flip'' argument as in~\cite{GR-WStabNew}. Applying the above approach of keeping track of the path of~$h(e)$ instead of forgetting both paths and choosing some path in the end would also make the ``spin flip'' argument in the proof of~\cite[Lemma~5.5]{GR-WStabNew} unnecessary.

\begin{proof} [Outline of the proof of Theorem~\ref{Thm6.3GRW}.]
    The proof of Theorem~\ref{Thm6.3GRW} is analogous to the proof of~\cite[Thm.~6.3]{GR-WStabNew}. The assumption of~$W$ being simply-connected is only used in~\cite[Lemma~6.8]{GR-WStabNew} so we just need to show that the map given in~\cite[Lemma~6.8]{GR-WStabNew} is $\bigl \lfloor \frac{g - \usr(\Z[\pi_1(W)]) - 1}{2}\bigr \rfloor$-connected for a~compact connected manifold~$W$ of dimension $2n \geq 6$ that is not necessarily simply-connected. But this follows from the proof of~\cite[Lemma~6.8]{GR-WStabNew} by using Theorem~\ref{ThmAllKConnectivity}~$(3)$ instead of~\cite[Cor.~5.10]{GR-WStabNew}.
\end{proof}

\begin{rmk}
    We can combine the above results with the results from Kupers in~\cite{KupersHomeo} for homeomorphisms, PL-homeomorphisms and homeomorphisms as a~discrete group of high-dimensional manifolds. Note that the machinery in Kupers' paper does not rely on the manifolds being simply-connected but rather the input does (i.e.\ the connectivity of a~certain complex uses that the manifold is simply-connected). Therefore, by using our more general theorem (Theorem~\ref{Thm7.4MvdK}) as the input, we can replace the assumption of the manifold being simply-connected by the group ring of the fundamental group having finite unitary stable rank.
\end{rmk}

\subsection{Tangential Structures and Abelian Coefficient Systems} \label{SecTangentalStructures}

In the remaining part of this chapter we extend Theorem~\ref{Thm6.3GRW} in two different ways. One is by considering moduli spaces of manifolds with some additional structure and the other is by taking homology with coefficients in certain local coefficient systems. We follow the approach of~\cite[Ch.~7]{GR-WStabNew}.

A~\textit{tangential structure} is a~map $\theta \colon B \rightarrow BO(2n)$, where $B$ is a~path-connected space. Let $\gamma_{2n} \rightarrow BO(2n)$ denote the universal vector bundle. A~$\theta$-structure on a~$2n$-dimensional manifold $W$ is a~bundle map (fibrewise linear isomorphism) $\hat \ell_W \colon TW \rightarrow \theta^*\gamma_{2n}$, with underlying map $\ell_W \colon W \rightarrow B$. We define $\calM^{\st, \theta}(W, \hat \ell_P)$ to be the set of manifolds $M \in \calM^{\st}(W)$ with $\partial M = P$ equipped with a~$\theta$-structure extending $\hat \ell_P$ for a~fixed pair~$(P, \hat \ell_P)$. As in the previous section we can also define the subset $\calM^{\theta}(W, \hat \ell_P) \subseteq \calM^{\st, \theta}(W \hat \ell_P)$ given by pairs~$(M, \hat \ell_M) \in \calM^{\st, \theta}(W, \hat \ell_P)$ with $M \in \calM(W)$. Using the topology described in~\cite[Ch.~7]{GR-WStabNew} and the correspondence
    \[\calM^{\st, \theta}(W, \hat \ell_P) = \bigsqcup_{[T]} \calM^{\theta}(T, \hat \ell_P),\]
where the union is taken over the set of compact manifolds~$T$ with $\partial T = P$, which are stably diffeomorphic to~$W$, one for each diffeomorphism class relative to~$P$, turns both sets into spaces.

We say that a~$\theta$-structure on~$\Sp^n \times \D^n$ is \textit{standard} if it is standard in the sense of~\cite[Def.~7.2]{GR-WStabNew}. The embeddings~$\overline e$ and $\overline f$ defined in Section~\ref{SecManifoldIntoQuadMod} yield embeddings
    \[e_1, f_1, \ldots, e_g, f_g \colon \Sp^n \longrightarrow W_{g, 1}.\]
We say that a~$\theta$-structure $\hat \ell \colon TW_{g, 1} \rightarrow \theta^* \gamma_{2n}$ on~$W_{g, 1}$ is \textit{standard} if there is a~trivialisation of the normal bundle of $\Sp^n$ (i.e.\ a~framing on~$\Sp^n$) such that the structures $\overline e_i^*\hat \ell$ and $\overline f_i^*\hat \ell$ on $\Sp^n \times \D^n$ are standard.

Let $\hat \ell_S$ be a~$\theta$-structure on the cobordism~$S \cong ([-1, 0] \times P) \# W_1$ which is standard when pulled back along the canonical embedding $\phi^{\prime} \colon W_{1, 1} \rightarrow S$. Writing $\hat \ell_P$ for its restriction to $\{0\} \times P \subset S$, and $\hat \ell^{\prime}_P$ for its restriction to $\{-1\} \times P$, we obtain the following map
\begin{eqnarray} \label{MapStabTheta}
    s = - \cup (S, \hat \ell_S) \colon \calM^{\st, \theta}(W, \hat \ell^{\prime}_P)
    & \longrightarrow
    & \calM^{\st, \theta}(W, \hat \ell_P) \\
    (M, \hat \ell_M)
    & \longmapsto
    & ((M - e_1) \cup S, \hat \ell_M \cup \hat \ell_S). \nonumber
\end{eqnarray}

As in~\cite{GR-WStabNew} we define the \textit{$\theta$-genus} for compact connected manifolds with $\theta$-structure as
    \[g^{\theta}(M, \hat \ell_M) = \max \left \{ g \in \N \ \bigg| \ \parbox{7cm}{there are $g$ disjoint copies of $W_{1, 1}$ in $M$, \ \linebreak each with standard $\theta$-structure} \right \}\]
and the \textit{stable $\theta$-genus} as
    \[\gbar^{\theta}(M, \hat \ell_M) = \max \{g^{\theta}\left((M, \hat \ell_M) \natural_k (W_{1, 1}, \hat \ell_{W_{1, 1}})\right) - k \ | \ k \in \N \},\]
where the boundary connected sum is formed with $k$ copies of $W_{1, 1}$ each equipped with a~standard $\theta$-structure $\hat \ell_{W_{1, 1}}$. As in the previous section, we can use the function $\gbar^{\theta}$ to consider $\calM^{\st, \theta}(W, \hat \ell_P)$ as a~graded space (in the sense of~\cite[Def.~6.6]{GR-WStabNew}) $\calM^{\st, \theta}(W, \hat \ell_P)_g$. With this notation, the stabilisation map~$s$ defined above then restricts to a~map
    \[s \colon \calM^{\st, \theta}(W, \hat \ell^{\prime}_P)_g \rightarrow \calM^{\st, \theta}(W, \hat \ell_P)_{g+1}.\]

We will now introduce a~class of local coefficient systems. Since the spaces considered here are usually disconnected and do not have a~preferred basepoint, local coefficients can be considered as a~functor from the fundamental groupoid to the category of abelian groups. Note that this is closely related to the corresponding definitions in~\cite{R-WWahl}. Then an~\textit{abelian coefficient system} is a~local coefficient system which has trivial monodromy along all nullhomologous loops.

\begin{thm} \label{Thm7.5GRW}
    Let $2n \geq 6$, $W$ be a~compact connected manifold, and $\calL$ be a~local coefficient system on~$\calM^{\st, \theta}(W, \hat \ell_P)$. Considering twisted homology with coefficients in~$\calL$ we get a~map
        \[s_* \colon H_k(\calM^{\st, \theta}(W, \hat \ell^{\prime}_P)_g; s^* \calL) \longrightarrow H_k(\calM^{\st, \theta}(W, \hat \ell_P)_{g+1}; \calL).\]
    \begin{enumerate}
        \item If $\calL$ is abelian then $s_*$ is an~epimorphism for $k \leq \frac{g - \usr(\Z[\pi_1(W)])}{3}$ and an~isomorphism for $k \leq \frac{g - \usr(\Z[\pi_1(W)]) - 3}{3}$.
        \item If $\theta$ is spherical in the sense of~\cite[Def.~7.4]{GR-WStabNew} and $\calL$ is constant, then $s_*$ is an~epimorphism for $k \leq \frac{g - \usr(\Z[\pi_1(W)])}{2}$ and an~isomorphism for $k \leq \frac{g - \usr(\Z[\pi_1(W)]) - 2}{2}$.
    \end{enumerate}
\end{thm}

For the case of simply-connected compact manifolds Galatius--Randal-Williams have shown in~\cite[Thm.~7.5]{GR-WStabNew} that the above stabilisation map~$s_*$ is an~isomorphism in a~range.

Given a~pair $(W, \hat \ell_W) \in \calM^{\theta}(W, \hat \ell^{\prime}_P)$, we write $\calM^{\theta}(W, \hat \ell_W) \subset \calM^{\theta}(W, \hat \ell^{\prime}_P)$ for the path component containing $(W, \hat \ell_W)$. By Theorem~\ref{Thm7.5GRW} the map
    \[s \colon \calM^{\theta}(W, \hat \ell_W) \longrightarrow \calM^{\theta} (W \cup_P S, \hat \ell_W \cup \hat \ell_S)\]
is an~isomorphism on homology with (abelian) coefficients in a~range of degrees depending on~$\gbar^{\theta}(W, \hat \ell_W)$.

The proof of Theorem~\ref{Thm7.5GRW} is analogous to the proof of Theorem~\ref{Thm6.3GRW}. We define a~quadratic module for a~pair $(W, \hat \ell_W) \in \calM^{\theta}(W, \hat \ell^{\prime}_P)$ as follows: Let $\calI^{\tn}_n(W, \hat \ell_W) \subseteq \calI^{\tn}_n(W)$ be the subgroup of those regular homotopy classes of immersions $i \colon \Sp^n \looparrowright W$ (together with a~path in~$W$) that have a~trivialisation of the normal bundle of~$\Sp^n$ such that the $\theta$-structure $i^* \hat \ell_W$ on $\Sp^n \times \D^n$ is standard. The bilinear form~$\lambda$ and the quadratic function~$\mu$ on~$\calI^{\tn}_n(W)$ restrict to the subgroup $\calI^{\tn}_n(W, \hat \ell_W)$ and hence define a~quadratic module $(\calI^{\tn}_n(W, \hat \ell_W), \lambda, \mu)$. As in the previous section this gives a~map of simplicial complexes
    \[K^{\delta}(W, \hat \ell_W) \longrightarrow \calH \calU(\calI^{\tn}_n(W, \hat \ell_W)),\]
where the complex $K^{\delta}(W, \hat \ell_W)$ is defined in~\cite[Def.~7.14]{GR-WStabNew}.

The following proposition is the analogue of Theorem~\ref{ThmAllKConnectivity}~$(3)$. For the case of simply-connected manifolds this has been shown in~\cite[Prop.~7.15]{GR-WStabNew}.

\begin{prop} \label{Prop7.15GRW}
    Let $2n \geq 6$, $W$ be a~compact connected $2n$-dimensional manifold, and $\hat \ell_W$ be a~$\theta$-structure on~$W$. Then the space $|\Kbar_{\bullet}(W, \hat \ell_W)|$ (defined in~\cite[Def.~7.14]{GR-WStabNew}) is $\bigl \lfloor \frac{\gbar(W, \hat \ell_W) - \usr(\Z[\pi_1(W)]) - 3}{2}\bigr \rfloor$-connected.
\end{prop}

\begin{proof} [Outline of the proof.]
    We have already seen in the previous section that an~embedding $i \colon W_{g, 1} \hookrightarrow W$ yields elements $e_1, f_1, \ldots, e_g, f_g \in \calI^{\tn}_n(W)$. If there is a~trivialisation of the normal bundle such that the $\theta$-structure $i^*\hat \ell_W$ is standard these elements are also contained in the subgroup $\calI^{\tn}_n(W, \hat \ell_W)$. In particular, we get $\gbar (\calI^{\tn}_n(W, \hat \ell_W)) \geq \gbar(W, \hat \ell_W)$ and the complex $\calH\calU(\calI^{\tn}_n(W, \hat \ell_W))$ is locally weakly Cohen-Macaulay (as defined in~\cite[Sec.~2.1]{GR-WStabNew}) of dimension $\bigl \lfloor \frac{\gbar(W, \hat \ell_W) - \usr(\Z[\pi_1(W)])}{2}\bigr \rfloor$ by Corollary~\ref{CorUConnectivityGbar}.

    We first show that the complex $|K^{\delta}(W, \hat \ell_W)|$ is $\bigl \lfloor \frac{\gbar(W, \hat \ell_W) - \usr(\Z[\pi_1(W)]) - 3}{2}\bigr \rfloor$-connected by arguing as in the proof of Theorem~\ref{ThmAllKConnectivity}~$(1)$. There we described how to get a~lift $F \colon I^{k+1} \rightarrow |K^{\delta}(M)|$ of the map
        \[\overline f \colon I^{k+1} \rightarrow |\calH\calU(I^{\tn}_n(W, \hat \ell_W))| \rightarrow |\calH\calU(\calI^{\tn}_n(W))|.\]
    As shown in the proof of~\cite[Prop.~7.15]{GR-WStabNew} we can turn this into a~lift $I^{k+1} \rightarrow |K^{\delta}(W, \hat \ell_W)|$.

    The connectivity of $|\Kbar_{\bullet}(W, \hat \ell_W)|$ now follows as in Theorem~\ref{ThmAllKConnectivity}.
\end{proof}

\begin{proof} [Outline of the proof of Theorem~\ref{Thm7.5GRW}.]
    This proof is is based on the proof of~\cite[Thm.~7.5]{GR-WStabNew} and we therefore just describe the changes that we have to make to that proof. Note that the simply-connected assumption is only used in~\cite[analogue of Lemma~6.8]{GR-WStabNew} so we only have to show that the map considered in that statement is $\bigl \lfloor \frac{\gbar^{\theta}(W, \hat \ell_W) - \usr(\Z[\pi_1(W)]) - 1}{2}\bigr \rfloor$-connected for a~compact and connected manifold $W$ of dimension~$2n\geq6$. But this follows analogously to the proof of~\cite[analogue of Lemma~6.8]{GR-WStabNew} by applying Proposition~\ref{Prop7.15GRW} instead of~\cite[Prop.~7.15]{GR-WStabNew} as in the original proof. This also explains the slightly lower bound in our case. Throughout this proof we need to replace \cite[Thm.~6.3]{GR-WStabNew} in the proof of~\cite[Thm.~7.5]{GR-WStabNew} by Theorem~\ref{Thm6.3GRW}.
\end{proof}

\bibliographystyle{plain}
\bibliography{C:/Users/Nina/Documents/MainBib}

\end{document}